\documentclass{birkjour}

\usepackage{amsmath, amssymb, amsthm, amscd, amsfonts, mathrsfs}
\usepackage[bookmarksnumbered,pdfpagelabels=true,plainpages=false,colorlinks=true,
            linkcolor=black,citecolor=black,urlcolor=blue]{hyperref}
\usepackage{cases}

\theoremstyle{plain}
\newtheorem{theorem}{Theorem}[section]
\newtheorem{lemma}[theorem]{Lemma}
\newtheorem{prop}[theorem]{Proposition}
\newtheorem{coro}[theorem]{Corollary}

\theoremstyle{definition}
\newtheorem{rema}[theorem]{Remark}
\newtheorem{notation}[theorem]{Notation}
\newtheorem{convention}[theorem]{Convention}
\newtheorem{defi}[theorem]{Definition}

\numberwithin{equation}{section}

\newcommand{\cF}{\mathcal F}

\newcommand{\cI}{\mathcal I}

\newcommand{\cL}{\mathcal L}
\newcommand{\cM}{\mathcal M}

\newcommand{\cP}{\mathcal P}

\newcommand{\cT}{\mathcal T}
\newcommand{\cU}{\mathcal U}

\newcommand{\ccP}{\mathscr{P}}

\newcommand{\al}{\alpha}
\newcommand{\be}{\beta}
\newcommand{\ga}{\gamma}
\newcommand{\Ga}{\Gamma}
\newcommand{\de}{\delta}

\newcommand{\la}{\lambda}
\newcommand{\La}{\Lambda}
\newcommand{\si}{\sigma}
\newcommand{\Si}{\Sigma}

\newcommand{\Om}{\Omega}

\newcommand{\RR}{\mathbb R}

\newcommand{\rar}{\rightarrow}

\newcommand{\tr}{\operatorname{tr}}
\newcommand{\ve}{\varepsilon}

\newcommand{\dive}{\operatorname{div}}

\newcommand{\p}{\parallel}

\newcommand{\K}{\mathscr{K}}
\newcommand{\ie}{\epsilon}
\newcommand{\Temp}{K}
\newcommand{\scr}{\varsigma}

\newcommand{\D}{\overline{\nabla}}
\newcommand{\mss}{\hspace{0.2cm}}
\newcommand{\ms}{\hspace{0.25cm}}
\newcommand{\msm}{\hspace{0.3cm}}
\newcommand{\msb}{\hspace{0.35cm}}

\newcommand{\Em}{\emph}
\newcommand{\fracbarA}{\frac{ \partial }{\partial x } {}_{ \hspace{-0.05cm} \bar{A} }}
\newcommand{\fracbarAline}{\frac{ \partial }{\partial x } {}_{ {}^{ \hspace{-0.05cm} \bar{A} } } }

\newcommand{\nosum}{ {\sum \hspace{-0.25cm} \big |} }

\newcommand{\wx}{\widetilde{x}}
\newcommand{\we}{\widetilde{e}}
\newcommand{\wGa}{\widetilde{\Gamma}}

\newcommand{\wk}{\widetilde{\kappa}}

\allowdisplaybreaks

\title[Einstein-Euler-Entropy system]{Remarks on the Einstein-Euler-Entropy system}

\author[Disconzi]{Marcelo M. Disconzi}
\address{Department of Mathematics\\
Vanderbilt University, Nashville, TN 37240, USA}
\email{marcelo.disconzi@vanderbilt.edu}

\begin{document}

\begin{abstract}
We prove short-time existence for the Einstein-Euler-Entropy system
for non-isentropic fluids with
data in uniformly local Sobolev spaces. 
The cases of compact as well as non-compact
Cauchy surfaces are covered.
The method employed uses a Lagrangian description of the fluid flow
which 
is based on techniques developed by Friedrich, hence providing 
a completely different proof of earlier results of 
Choquet-Bruhat and Lichnerowicz. This new proof  is 
specially suited for applications to self-gravitating fluid bodies.
\end{abstract}

\maketitle

\tableofcontents


\section{Introduction. \label{intro}}

The Einstein equations have been a source of several 
interesting problems in Physics, Analysis and Geometry. 
Despite the great deal of work that has been devoted to them, 
with many success stories, several important questions remain
(see \cite{CGP, Da, Schmidt_ed} for an account
of what is currently known and some 
directions of future research).
One of  them is finding a satisfactory theory of isolated systems, 
such as stars, both from a perspective of the time development of 
the space-time, as well as from the point of view of the 
geometry induced on a space-like three surface. 
To quote Rendall,
``of the physical situations which can be described by the general theory of relativity, those which are 
at the present most accessible to observation are the isolated systems. In fact, all existing tests of
Einstein's equations concern such situations. It is therefore important to have a theory of these systems 
which is as complete as possible, not only in terms of the range of phenomena which are covered but also 
with respect to logical and mathematical solidity'' \cite{Re}.

Stars are the prototypes of isolated systems. They are typically modeled by 
considering a region of space filled with a fluid and separated from an 
exterior that corresponds to vacuum. The properties
of the fluid, such as perfection, viscosity, charge, etc., depend on the particular
situation one is interested in. The dynamics of the fluid region is then described 
by Einstein equations coupled to matter, whereas vacuum 
Einstein equations hold on the complement of this set.
From the point of view of the Cauchy problem, which will be the main case of interest in 
this work, there  are  two primary questions to be addressed.

(i) First, the solvability of Einstein equations in the two different situations
of interest, i.e., coupled to fluid sources and in vacuum, should  be addressed.
Since the short-time existence for vacuum Einstein equations is well 
understood (see e.g. \cite{HE}), this leaves us with the coupling to matter.
A specific matter model has then to be chosen, and one of the most 
customary choices, sufficient for many applications, is that of a perfect fluid
\cite{Anile}, in which case Einstein equations are coupled to the (relativistic)
Euler equations. The well-posedness\footnote{``Well-posedness'' should be understood here
in the context of General Relativity, where uniqueness is meant in a geometric sense, i.e.,
up to isometries.}
 of this system was proven by 
Choquet-Bruhat \cite{C2} and extended by Lichnerowicz \cite{Lich} to include 
entropy --- in which case the resulting system will be called the Einstein-Euler-Entropy 
system, whose equations are stated in section \ref{EEE_system_section}.

(ii) The second question is more delicate and consists of trying to bring 
together the two different scenarios described in the previous paragraph, namely,
vacuum and coupling to matter. More precisely, we attempt to formulate and solve
the Cauchy problem with an interior region --- thought of as the star --- governed by the 
Einstein-Euler (or Einstein-Euler-Entropy) equations
and an exterior one that evolves according to the vacuum Einstein equations. 
This corresponds to a genuine free-boundary problem in that the boundary of the star 
cannot be prescribed for time $t>0$, being rather a dynamic quantity that has to be determined
from the evolution equations. The star boundary at time zero is obtained
as the boundary of the support $\Om$ of the initial matter density $\varrho_0$, which is 
a function on the initial Cauchy surface typically of the form
\begin{gather}
\varrho_0(x) =
 \begin{cases}
  f(x) > 0, &  x \in \Om, \\
 0, & x \notin \Om.
 \end{cases}
\label{intro_matter_vanish}
\end{gather}
Solving the Cauchy problem then requires a refined
analysis of the boundary behavior of the quantities involved, where the 
change from $G_{\al\be} = T_{\al\be}$ to $G_{\al\be} = 0$ causes 
severe technical difficulties (here $G_{\al\be}$ and $T_{\al\be}$ are the 
Einstein and stress-energy tensor, respectively).

Lindbom \cite{Lindblom} has proven that any isolated static fluid stellar model 
ought to be spherically symmetric, generalizing a classical  result of
Carleman and Lichtenstein for Newtonian fluids \cite{Carl, Lichten1, Lichten2}.
More precisely, he has shown that 
a  static
asymptotically flat space-time that contains only a uniform density perfect fluid
confined to a  spatially compact world tube is necessarily spherical symmetric.

Under the assumption of spherical symmetry, it is possible to deal
with many of the technicalities that  arise, and a number of satisfactory and general 
results
have been obtained.
 Rendall and Schmidt 
have proven existence and uniqueness results for global solutions of 
the Einstein-Euler system \cite{ReSch}. They also gave
 a detailed account of how the properties
of these solutions depend on features of the equation of state. 
Kind and Ehlers have treated the mixed initial-boundary value problem in \cite{KindEhlers}, 
where they have also given necessary and sufficient conditions for attaching the solutions 
they  
constructed to a Schwarzschild space-time.
Makino \cite{MakinoSpherical} refined the results of Rendall and Schmidt by 
providing  a general criterion for the equation of state, which ensures that the model
has finite radius (and therefore finite mass). This result is used
to study the linear stability of the equations of motion.
Regarding the relation between mass and radius, in \cite{HRU} the authors 
derive, among other results, interesting mass-radius theorems.
A very extensive treatment of the Cauchy problem for spherically symmetric data
is given in the work of Groah, Smoller, and  Temple \cite{GST}. These works 
do not exhaust all the results known in spherical symmetry; see the references in the above
papers for more details. 

If we drop the hypothesis of spherical symmetry, however, not much is 
known about the well-posedness of the Cauchy problem when the matter 
density is allowed to vanish outside a compact set, as in (\ref{intro_matter_vanish}) 
--- a situation generally referred to as a ``fluid body''. One could attempt to change $\varrho_0$ to a function
that is always positive, but decays sufficiently fast and is very small in the region
outside $\Om$. Unfortunately, this does not improve things considerably because 
generally the time-span of solutions cannot be shown to be uniform, and as a result,
the domain of definition of the solution contains 
no space-like slice $t=\,$constant (except obviously the $t=0$ slice). 
Another problem which arises in this context is that the usual conformal method
for solving the constraint equations \cite{LichConf, YorkConf, CorvinoPollackConf} cannot be applied.

Following the ideas of Makino \cite{Mak}, who studied non-relativistic\footnote{It
should be noticed that even self-gravitating Newtonian fluid bodies 
are not well understood. In fact, with the exception of the viscous case studied
by Secchi \cite{Se1, Se2, Se3}, not many results besides that of Makino \cite{Mak} seem to be available.} 
gaseous stars,
Rendall has shown that the above problems can be circumvented for fluids obeying
certain equations of state \cite{Re2} --- although, as the author himself points out,
the solutions obtained in these cases have the undesirable property of not including 
statically spherically symmetric space-times, and the 
restrictions on the equation of state are too strong.
Initial data sets for the Einstein-Euler equations describing a fluid body
have been successfully constructed by Brauer and Karp \cite{BrauerKarp1}
and Dain and Nagy \cite{DainNagy}, but 
well-posedness of the system with such prescribed initial values has not yet 
been demonstrated. Therefore, the solution to the Cauchy problem for a 
self-gravitating isolated fluid body is still 
largely open (although some recent results of Brauer and
 Karp \cite{BrauerKarp2, BrauerKarp3, BrauerKarp4} 
address directly some of the technical issues we referred to).

An important step forward has been achieved by Friedrich in \cite{Fri}. Using
a Lagrangian description of the fluid motion, he has been able to derive 
a set of reduced equations which form a first order symmetric hyperbolic system.
While it has been known that the vacuum Einstein equations can be cast 
in such a form since the work of Fischer and Marsden \cite{FischerMarsden}, and 
the Einstein-Euler system had also been investigated within the formalism of 
first order symmetric hyperbolic systems 
in the aforementioned work of Rendall, what makes Friedrich's construction
particularly attractive is the use of the Lagrangian description of the fluid flow,
as it is known that Lagrangian coordinates are uniquely suited
for treating free boundary problems involving the non-relativistic 
Euler equations, and have in fact been employed to a great success 
to study them \cite{E1, CS, ShaZen}.

It remains to be seen whether Friedrich's ideas will 
lead to satisfactory existence theorems for the 
Cauchy problem for fluid bodies. 
But for such a project to be successful, one needs to be able 
to carry out step (i) as explained above, namely, to use 
the reduced system of \cite{Fri} to separately solve Einstein equations 
in the cases of vacuum and coupled to fluid matter. Only then 
the passage across the boundary from the Einstein-Euler-Entropy to 
the vacuum equations can be analyzed. Notice that in the case of
vacuum, the notion of a Lagrangian description is somewhat artificial,
but it can be given meaning by the choice of a time-like vector
field which should agree with the fluid four-velocity once matter
is introduced.

Well-posedness of the vacuum Einstein equations using the formalism
of \cite{Fri} has been established in 
\cite{FriHypRed, FriHypGauge} (see also \cite{Friedrich1, Friedrich2, Friedrich3}).
This leaves us with the question of well-posedness of the Einstein-Euler-Entropy system. 
In other words, we seek a solution to the following:

\emph{Establish a well-posedness result for the Einstein-Euler-Entropy system with 
a general equation of state, by writing the set of equations as a first order symmetric
hyperbolic system via a Lagrangian description of the fluid flow,
as in \cite{Fri}, and under the assumption that the fluid fills the entire 
initial Cauchy surface $\Si$, i.e. $\varrho_0 \geq c > 0$, with possibly 
additional hypotheses consistent with physical requirements.}

This is the problem addressed and
solved in this work. Some terminology is needed before we can make 
a precise statement --- see 
theorem \ref{main_theorem} in section \ref{EEE_system_section}.
We stress that theorem \ref{main_theorem} had been proven much earlier
by Choquet-Bruhat \cite{C2} and Lichnerowicz \cite{Lich}. What is new is 
the method of proof, employing the Lagrangian description 
of the fluid flow. It should also be noticed that, while the equations we use
are essentially those of \cite{Fri} and \cite{FriRen}, to the best of 
our knowledge, a complete proof of well-posedness relying on 
Friedrich's techniques is not available in the literature.

\subsection{Outline of the paper and notation.} The 
paper is organized as follows. In section \ref{basic_setting},
we recall several basic definitions, fix our notation, introduce the main
equations and hypotheses, and state the main theorem. In section 
\ref{frame_formalism}, we introduce the frame formalism and gauge 
conditions that constitute the basis of Friedrich's method. 
The gauged or reduced Einstein-Euler-Entropy system is also presented in this section.
The proof of the main theorem is carried out in section 
\ref{proof_main_theorem_section}. It consists of three parts:
determination of the initial data (section \ref{initial_data_section}),
well-posedness of the reduced system (section \ref{well_posedness_reduced_section}),
and the ``propagation of the gauge'' (section \ref{propagation_gauge}).
This last step is what guarantees that a solution of the Einstein-Euler-Entropy
system in a particular gauge --- the reduced system --- yields 
a solution to the original set of equations. Finally, in section 
\ref{further}, we make some closing remarks.

We have chosen to present our arguments in a logical
rather than constructive order. This means that
instead of starting by showing how the reduced equations are obtained from the 
Einstein-Euler-Entropy system by means of a gauge choice, we first 
state the reduced equations, then derive its solutions, and finally show
that they correspond to solutions of the original equations of motion.
This particular order is adopted because the equations we 
use in both the reduced and the propagation
of the gauge systems are equivalent to those of \cite{Fri} and \cite{FriRen}, 
although we shall write them in a slightly different fashion. 
Hence, as they have been considered before, there 
is no immediate need to show how the reduced
equations are extracted from the original ones by a suitable choice of
gauge. This procedure is briefly presented nonetheless in the appendix
for the reader's convenience. Those not familiar with the work \cite{Fri}
are encouraged to read  appendix \ref{derivation_appendix} prior to 
section \ref{proof_main_theorem_section}.

\begin{notation}
The Sobolev space between manifolds $M$ and $N$ will be denoted by $H^s(M,N)$. 
When no confusion can arise, the reference to $N$ will be suppressed, and 
$H^s(M,N)$ will be written $H^s(M)$, or even $H^s$ when $M$ is clear from the context.
$H^s_{ul}$ denotes the uniformly local Sobolev spaces, whose definition is recalled 
in appendix \ref{ul_Sobolev_appendix}.
\end{notation}

\section{The basic setting and the main result\label{basic_setting}.}
Our main object of study is a four-dimensional Lorentzian manifold $(\cM,g)$,
called a \Em{space-time}.  In General Relativity, 
we are interested in the sub-class of Lorentzian manifolds where 
Einstein equations are satisfied. We recall that these are
\begin{align}
R_{\al \be} - \frac{1}{2} R g_{\al \be} =\K  T_{\al \be},
\label{Einstein_eq}
\end{align}
where $R_{\al\be}$ and $R$ are respectively the Ricci and scalar curvature of the metric $g$, 
 $T_{\al\be}$ is the stress-energy tensor which encodes information about the matter fields,
 and $\K =8 \pi \frac{G}{c^4}$, with $c$ being the speed of light and $G$ Newton's
constant\footnote{As it is customary to drop the factor $8 \pi$, and we adopt  units  where
$G=c=1$, we could set $\K=1$. But since we will be working with a lesser familiar formalism, it
is useful to keep the constant $\K$, in which case it is always possible to set $\K = 0$
in order to compare the resulting equations with the well studied vacuum case.}. A
space-time is called an \Em{Einsteinian space-time} when (\ref{Einstein_eq}) is satisfied.
$\cM$ will always be assumed to be oriented and time-oriented.

The left-hand side of (\ref{Einstein_eq}) is divergence free as a consequence of the Bianchi
identities. Hence, regardless of the particular matter model which is considered, the stress energy tensor has to satisfy
\begin{align}
\nabla^\al T_{\al \be} =0,
\label{conservation_T}
\end{align}
which then gives equations of motion for the matter 
fields (see section \ref{EEE_system_section} below); $\nabla$ in (\ref{conservation_T})
denotes the Levi-Civita connection associated with $g$. Equation (\ref{conservation_T}) is
sometimes referred to as the local law of momentum and energy conservation.

\begin{convention} Throughout the paper we shall adopt the convention $(+ \, - \, - \, - )$
for the metric. 
\end{convention}

\begin{defi}  A \Em{fluid source} is a triple $(\cU, g, u)$, where $(\cU,g)$ 
is a domain of a space-time and $u$ a time-like vector field on $\cU$ of unit norm. Physically, 
the trajectories of $u$ represent the flow lines of matter. Given a fluid source, the 
\Em{stress-energy tensor of a perfect fluid} is given by
\begin{align}
T_{\al\be} = (p + \varrho) u_\al u_\be - p g_{\al\be},
\label{perfect_fluid_source}
\end{align}
where $p$, called the \Em{pressure of the fluid}, and $\varrho$, called
the \Em{density of the fluid}, are non-negative real valued functions. A \Em{perfect fluid source}
is a fluid source together with a stress-energy tensor given by (\ref{perfect_fluid_source}).
\end{defi}
Perfect fluid sources, or perfect fluids for short, are 
often used to study space-times where a continuous distribution 
of matter exists (see e.g. \cite{WeinbergGR}). 
The assumption that matter is described by a stress-energy tensor of the form
(\ref{perfect_fluid_source}) means that no dissipation of any sort is present; in
particular one neglects possible effects due to heat conduction, 
viscosity or shear stresses\footnote{In fact, the (here postulated) stress-energy 
tensor (\ref{perfect_fluid_source}),
can be motivated from basic physical principles, which can be taken as 
characterizing ideal fluids; see e.g. \cite{Anile,Dixon}.}.

For perfect fluids, equation (\ref{conservation_T})
becomes
\begin{align}
\begin{split}
\nabla^\al T_{\al\be} = & \, \Big ( (p + \varrho ) \nabla^\al u_\al + u^\al \nabla_\al \varrho  \Big ) u_\be \\
& + (p + \varrho) u^\al \nabla_\al u_\be + u_\be u^\al \nabla_\al p - \nabla_\be p = 0.
\end{split}
\label{conservation_T_perfect}
\end{align}
Taking the inner product of (\ref{conservation_T_perfect}) with $u$ and using 
$u_\al u^\al = 1$ (which also implies $u^\al \nabla_\be u_\al = 0$), gives 
\begin{gather}
 (p + \varrho ) \nabla^\al u_\al + u^\al \nabla_\al \varrho = 0,
\label{Euler_1}
\end{gather}
known as the \Em{conservation of energy} or \Em{continuity equation}. Then, using
(\ref{Euler_1}) into (\ref{conservation_T_perfect}) produces
\begin{gather}
 (p + \varrho) u^\al \nabla_\al u_\be + u_\be u^\al \nabla_\al p - \nabla_\be p = 0,
\label{Euler_2}
\end{gather}
known as \Em{conservation of momentum equation}. Together, equations
(\ref{Euler_1}) and (\ref{Euler_2}) are known as the \Em{Euler equations} for a relativistic fluid.

In physically relevant models, the functions
$p$ and $\varrho$ are not independent but are related by what is called
an \Em{equation of state}, which is an (usually smooth) invertible function 
\begin{gather}
 p = p(\varrho).
\label{barotropic}
\end{gather}
Fluids where (\ref{barotropic}) is satisfied are called \Em{barotropic} fluids, with
the particular case $p \equiv 0$ called \Em{dust} or \Em{pressure-free matter}. Physical 
situations of interest where barotropic fluids are employed include 
some models of cold (more precisely, zero temperature) matter, such 
as completely degenerate cold neutron gases (which are used to model nuclear matter in the interior 
of neutron stars) \cite{C, Anile, ShapiroTeukolsky}; the so-called ultra-relativistic fluids, i.e., 
fluids in thermal equilibrium where the energy 
(which in relativistic terms is described by $\varrho$ due to the equivalence of 
mass and energy) is largely dominated by radiation \cite{Anile, WeinbergGR}; and fluids of 
electron-positron pairs \cite{WeinbergGR}. Barotropic fluids are also important 
in Cosmology, where several models of the early universe assume that the distribution
of matter is described by an ultra-relativistic fluid \cite{C, Anile, WeinbergGR, WeinbergCosmology}.

As usual in General Relativity, space-time itself is a dynamic quantity, therefore
its geometry, encoded in the metric $g$, as well as the the dynamics of other fields present
on space-time, have to be determined as solutions to the Einstein equations (\ref{Einstein_eq}) coupled
to matter via (\ref{conservation_T}). In particular, the fluid source $(\cU,g,u)$ is not a given, but
has to arise from the solutions of the coupled system. The 
\Em{Einstein-Euler system for barotropic fluids}, or \Em{barotropic Einstein-Euler system},  
is given by equations (\ref{Einstein_eq}), (\ref{Euler_1}), (\ref{Euler_2}), 
and (\ref{barotropic}), with $T_{\al\be}$ given by (\ref{perfect_fluid_source}),
and subject to the constraint $u^\al u_\al = 1$. The unknowns to be determined 
are the metric $g$, the
four-velocity $u$, and the matter density $\varrho$.

The Einstein-Euler system for barotropic fluids has been studied by many authors.
In fact, most of the results cited in the 
introduction deal with this case. The interested reader can consult
the papers \cite{DainNagy, BrauerKarp1, Re2, HRU, GST, MakinoSpherical, ReSch} 
and the monographes \cite{Lich, Anile, C}, as well as the references therein.

Similar to the well-known case of vacuum Einstein 
equations \cite{C_thesis}, when addressing the 
solvability of the barotropic Einstein-Euler system, one has to investigate its constraints and 
make suitable choices for the spaces where solutions will be sought. Since we shall deal with the more 
general case of fluids that are not necessarily barotropic, we postpone this discussion for
the time being, turning our attention first to some thermodynamic tools that will be 
necessary in the sequel. We shall return to the barotropic case in section \ref{further}.

\subsection{Thermodynamic properties of perfect fluids.}
There are important scenarios where it is known that the pressure is 
not determined by the matter density only. 
These include the so-called polytropic fluids, which
are used in several stellar models 
(see e.g. \cite{C, Anile, WeinbergGR, ZelNov} and references therein).
In such cases, (\ref{barotropic}) 
has to be replaced by a more general equation of state involving 
other physical quantities, which are usually assumed to be of 
thermodynamic nature, as we now describe.

As in many situations in General Relativity, it is important to identify
those quantities measured by a local inertial observer. Let $r$ be the 
\Em{rest mass (or energy)\footnote{Due to the equivalence of mass and energy, and our choice
of units with $c=1$, we shall use the terms mass and energy interchangeably.}
 density}\footnote{Also called \Em{particle number density} \cite{FriRen}, 
\Em{baryon number density} \cite{Anile} or yet \Em{proper material density}
 \cite{Lich}.}, defined as the mass density measured in the 
local rest frame. It is assumed that this quantity obeys the conservation law
\begin{gather}
 \nabla_\al(r u^\al) = 0,
\label{rest_mass_conservation}
\end{gather}
which states that mass is (locally) conserved. Notice that on dimensional grounds,
$\frac{1}{r}$ is the specific (i.e., per unit of mass) volume. The difference between 
the mass density $\varrho$ and the rest mass density $r$ is 
by definition the \Em{specific  internal energy} $\ie$
(as measured in the local rest frame):
\begin{gather}
\varrho = r(1 + \ie). 
\label{internal_energy}
\end{gather}
Therefore, as $p$, $r$ and $\varrho$, $\ie$ is a real valued function on $\cU \subseteq \cM$. 
In the context of relativistic fluids, the 
relation (\ref{internal_energy}) has been first introduced by Taub
\cite{Taub} and has been widely used since \cite{FriRen,C,Anile,Lich}.

For non-barotropic fluids, further relations among 
the variables of the problem have to be introduced in order 
to have a well-determined system of equations. It is natural to assume that 
the first law of thermodynamics holds, i.e., 
\begin{gather}
 d \ie = \Temp \, ds -  p \, dv,
\nonumber
\end{gather}
where $\Temp$ is the absolute 
temperature, $s$ the specific entropy, and $v$ the specific volume.
These are all non-negative real valued functions on $\cU$. 
In light of (\ref{internal_energy}) and $v=\frac{1}{r}$, the first law can 
be written as\footnote{The case of interest in this paper is when the matter density does
not vanish, hence $\frac{1}{r}$ is well defined because of (\ref{internal_energy}). The case 
of vanishing $\varrho$ is, however, important, as explained in the introduction. }
\begin{gather}
 d \varrho = \frac{p + \varrho}{r} \, dr + r \Temp \, ds.
\label{first_law}
\end{gather}
We now turn to the generalization of (\ref{barotropic}), i.e., to the 
appropriate equation of state that has to be provided.
Although we presently have seven thermodynamic variables, namely, 
$p$, $\varrho$, $r$, $\ie$,  $\Temp$, $v$ and $s$, the above equations imply 
relations among them. In fact, for a perfect fluid, only two of such quantities
are independent \cite{Anile}, with the remaining ones determined by relations depending
exclusively on the nature of the fluid. On physical grounds, we should assume that 
such relations are invertible, what renders the question of which two thermodynamical
quantities are the independent ones a matter of preference. We shall assume henceforth that 
$r$ and $s$ are independent. We thus postulate an equation of state of the form
\begin{gather}
 \varrho = \ccP(r,s),
\label{eq_of_state}
\end{gather}
where $\ccP$ is a given smooth function, invertible in the sense that we can solve for 
$r = r(\varrho,s)$ and $s = \varrho(\varrho, r)$.
From (\ref{first_law}) and (\ref{eq_of_state}) we then obtain
\begin{gather}
 p = r \frac{\partial \varrho}{\partial r} - \varrho,
\label{pressure}
\end{gather}
and 
\begin{gather}
 \Temp = \frac{1}{r} \frac{\partial \varrho}{\partial s}.
\label{temperature}
\end{gather}
The validity of (\ref{first_law}) along with (\ref{Euler_1}) implies
\begin{gather}
 \Temp u^\al \nabla_\al s = -\frac{1}{r^2} (p + \varrho) \nabla_\al( r u^\al).
\label{derivation_loc_adiabatic}
\end{gather}
Therefore, in light of (\ref{rest_mass_conservation}), and assuming that the temperature is not zero
(which is consistent with (\ref{temperature}) and (\ref{eq_of_state})), 
\begin{gather}
 u^\al \nabla_\al s = 0.
\label{loc_adiabatic}
\end{gather}
In other words, the entropy is conserved along the flow lines of the fluid. When
(\ref{loc_adiabatic}) holds, the motion of the fluid is said to be 
\Em{locally adiabatic}. A fluid is said to be \Em{isentropic} if $s=$constant, and 
\Em{non-sentropic} otherwise. 

\begin{rema}
 If (\ref{rest_mass_conservation}) is not assumed, then from (\ref{derivation_loc_adiabatic}) 
it only follows that
\begin{gather}
 \nabla_\al(r u^\al) \leq  0,
\nonumber 
\end{gather}
and
\begin{gather}
 u^\al \nabla_\al s \geq 0.
\nonumber 
\end{gather}
Although these two inequalities have a clear physical 
interpretation --- from the point of view of 
inertial  observers the rest mass cannot increase and the entropy cannot decrease ---
if (\ref{rest_mass_conservation}), 
and therefore (\ref{loc_adiabatic}), is not assumed, the motion of the fluid 
is underdetermined. 
\end{rema}
The first law of thermodynamics, equation (\ref{first_law}), will  always be
 assumed to hold, 
 therefore, because of (\ref{derivation_loc_adiabatic}),  we shall work with (\ref{loc_adiabatic})
rather than (\ref{rest_mass_conservation}) to form our system of equations.

\subsection{The Einstein-Euler-Entropy system.\label{EEE_system_section}}
The \Em{Einstein-Euler-Entropy system} is the system comprised of equations
(\ref{Einstein_eq}), 
(\ref{Euler_1}), (\ref{Euler_2}), (\ref{eq_of_state}), (\ref{pressure}), 
(\ref{loc_adiabatic}),
subject to the constraint
$u^\al u_\al = 1$, with $T_{\al\be}$ given by (\ref{perfect_fluid_source}), and $\ccP$ in
(\ref{eq_of_state}) a given smooth invertible function. The unknowns to be 
determined are the metric $g$, the four-velocity $u$, 
the rest mass density $r$, and the specific entropy $s$.

For previous works on the Einstein-Euler-Entropy system, 
the reader can consult the monographes \cite{Lich, Anile, C}.

Our focus here is on the Cauchy  problem, therefore we need  to state
what the initial data are. This should consist of the usual initial
data for the Einstein equations and initial data for the matter fields. These have 
to satisfy suitable constraint equations, as we now recall. Because of our signature convention,
the metric $g_0$ on the initial three slice $\Si$ is negative-definite, a fact which we stress
by calling the pair $(\Si, g_0)$ a \Em{negative Riemannian manifold}.

\begin{defi} A \Em{pre-initial data set for the Einstein equations},
or \Em{pre-initial data set} for short, is a triple 
$(\Si, g_0, \kappa)$, where $(\Si, g_0)$ is a three-dimensional negative Riemannian 
manifold and $\kappa$ a symmetric two-tensor on $\Si$. A \Em{development}
of a pre-initial data set $(\Si, g_0, \kappa)$ is a space-time $(\cM,g)$ which admits
an isometric embedding of $(\Si, g_0)$, with $\kappa$ being the second fundamental form
of the embedding. A development $(\cM,g)$ is called an \Em{Einsteinian development}
if $(\cM, g)$ is an Einsteinian space-time, i.e., Einstein equations are satisfied on $\cM$.
\label{development}
\end{defi}

Let $(\cM, g)$ be an Einsteinian development of $(\Si, g_0, \kappa)$. Then 
the fact that $\Si$ is embedded into $\cM$ with 
second fundamental form $\kappa$ and Einstein equations are satisfied implies
that the following identities hold\footnote{See e.g. \cite{HE} for a proof. The plus sign 
on $\mu$ is due to our signature convention; some authors define $J$ with a negative sign.} on $\Si$,
\begin{gather}
R_{g_0} - |\kappa |_{g_0}^2 + (\tr_{g_0} \kappa)^2 + 2  \mu= 0,
\label{hamilonian_constraint}
\end{gather}
known as the \Em{Hamiltonian constraint}, and 
\begin{gather}
\dive_{g_0} \big (\kappa - (\tr_{g_0}  \kappa )  g_0\big  ) +  J = 0,
\label{momentum_constraint}
\end{gather}
known as the \Em{momentum constraint}, with $\mu$ and $J$  defined  by
\begin{gather}
\mu := \K T(n, n),
\label{definition_mu}
\end{gather}
and
\begin{gather}
J := \K T(n, \cdot),
\label{definition_J}
\end{gather}
where $n$ is the unit normal of $\Si$ inside $\cM$,  $T$ is the stress-energy tensor,
and $T(n, \cdot)$ is viewed as a one-form on $\Si$.
In the above, $R_{g_0}$, $| \cdot|_{g_0}$, $\tr_{g_0}$, $\dive_{g_0}$ are, respectively, the scalar curvature, the pointwise norm, the trace, and the divergence, all 
with respect to the metric $g_0$. 
Notice that definition \ref{development} and equations
(\ref{hamilonian_constraint}), (\ref{momentum_constraint}), (\ref{definition_mu}),
 (\ref{definition_J}) are general, in the sense that they do not assume that $T$
is the stress-energy tensor of a perfect fluid. When $T$ has the form 
(\ref{perfect_fluid_source}), then (\ref{definition_mu}) and (\ref{definition_J}) become
\begin{gather}
\mu = \K (p + \varrho)(1-|\pi(u)|_{g_0}^2) - \K p,
\label{definition_mu_perfect}
\end{gather}
and\footnote{As usual we identify vectors and co-vectors, since 
$J$ is given as a vector in (\ref{definition_J_perfect}) but as a one form in (\ref{definition_J}).}
\begin{gather}
J = \K  (p + \varrho)  \sqrt{1-|\pi_g(u)|_{g_0}^2 }\pi_g(u),
\label{definition_J_perfect}
\end{gather}
where $\pi_g:  \left. T\cM \right|_\Si \rar T\Si$ is the orthogonal projection onto the
tangent bundle of $\Si$.

From (\ref{definition_mu_perfect}) and (\ref{definition_J_perfect}) it is seen that, 
additionally to $(\Si, g_0, \kappa)$, to solve  the Einstein-Euler-Entropy system, one
has to prescribe $\varrho$, $p$, and a vector field $v$ (which will satisfy $v=\pi(u)$ 
once solutions are obtained). However, as we adopt the point of view 
that the independent thermodynamic
variables are $s$ and $r$, we do not prescribe $p$ and $\varrho$ directly. Rather, 
$r$ and $s$ are given as initial data, and then $\varrho$ and $p$ are determined 
on $\Si$ by
(\ref{eq_of_state}) and (\ref{pressure}), respectively. 

\begin{defi}
An \Em{initial data set for the Einstein-Euler-Entropy system}
is a $7$-uple 
$(\Si, g_0, \kappa$, $r_0, \scr_0, v, \ccP)$, where
$(\Si, g_0)$ is a three-dimensional negative Riemannian manifold
endowed with a symmetric 
two-tensor  $\kappa$;  $r_0$ and $\scr_0$ are non-negative real valued 
functions\footnote{We write $\scr_0$ instead of $s_0$ for the entropy at time zero
to avoid confusion with the quantity $s_\al$ introduced below.}
$r_0, \scr_0: \Si \rar \RR_+$; $\ccP: \RR \times \RR \rar \RR_+$
is a smooth invertible function as in (\ref{eq_of_state});
and  $v$ is a vector field on $\Si$;
such that the \Em{constraint equations} (\ref{hamilonian_constraint}) and (\ref{momentum_constraint}) are 
satisfied, with $\mu$ and $J$ given by
\begin{gather}
\mu = \K (p_0 + \varrho_0)(1-|v|_{g_0}^2) - \K p_0, 
\nonumber
\end{gather}
and
\begin{gather}
J = \K  (p_0 + \varrho_0)  \sqrt{1-|v|_{g_0}^2} \, v,
\nonumber 
\end{gather}
where $\varrho_0 = \varrho(r_0,\scr_0)$ and $p_0 = p_0(r_0,\scr_0)$  are given by
(\ref{eq_of_state}) and (\ref{pressure}), respectively.
\end{defi}
In order to state the main result, we need to impose further conditions that are either
of physical nature or are required in order to apply certain analytic techniques.

One quantity of physical relevance in non-relativistic fluid dynamics is the speed of
acoustic waves. For relativistic fluids, we define the \Em{sound speed} $\nu$ by
\begin{gather}
 \nu^2 = \left( \frac{\partial p}{\partial \varrho } \right)_s = \frac{ r }{ p + \varrho} \frac{ \partial p }{\partial r},
\label{sound_speed}
\end{gather}
where, as usual in Thermodynamics, $\left( \frac{\partial }{\partial x} \right)_y$ means 
treating $y$ as a constant when taking derivatives with respect to $x$. 
$\nu^2$ is well defined in that, for physically relevant equations of state, pressure
cannot decrease as $\varrho$ increases.
The fact 
that (\ref{sound_speed}) is correctly interpreted as the speed of acoustic waves
follows from analyzing the characteristic manifolds of the system formed by equations
(\ref{Euler_1}), (\ref{Euler_2}), (\ref{loc_adiabatic}) and (\ref{eq_of_state}) with a given
background metric; we refer the reader to \cite{Anile, Lich} for details. 

We shall require that the following inequalities hold:
\begin{subequations}{\label{positivity_conditions}}
\begin{align}
 \nu^2 > 0,
\label{sound_speed_positive} \\
\nu^2 \leq 1, 
\label{sound_speed_causal} \\
\frac{p+ \varrho}{r} > 0.
\label{enthalpy_positive}
\end{align}
\end{subequations}
(\ref{sound_speed_positive}) and (\ref{sound_speed_causal}) assert 
the physical requirements that the sound
speed is positive and does not exceed the speed of light\footnote{It is not difficult
to see that equality in (\ref{sound_speed_causal}) happens if, and only if, the four-divergence 
of $u$ vanishes, i.e., $\nabla_\al u^\al =0$ (see e.g. \cite{C}). When this happens, the fluid is called \Em{incompressible}. Recall
that incompressibility is defined in non-relativistic physics as the vanishing of the three-velocity, in which
case sound waves travel with infinity speed.}. 
The interpretation of (\ref{enthalpy_positive}) is that \Em{specific enthalpy} of the fluid, defined as
$\frac{p+\varrho}{r}$, is positive. 
In view of (\ref{sound_speed}), it should be noticed 
that (\ref{sound_speed_positive}) together 
with (\ref{enthalpy_positive}) excludes the
possibility of pressure-free matter. Our techniques can be easily adapted
to allow such a case nonetheless, leading to a theorem similar to \ref{main_theorem} below 
(see section \ref{further}).

When the initial slice $\Si$ is compact and the functions involved continuous,
inequalities (\ref{sound_speed_positive}) and (\ref{enthalpy_positive}) automatically imply
that on $\Si$
\begin{gather}
\nu^2 \geq c,\, \text{ and } \frac{p+ \varrho}{r} \geq c,
\label{bound_from_below}
\end{gather}
for some constant $c>0$. These conditions are needed 
to apply certain existence results for first order symmetric hyperbolic systems. When $\Si$
is non-compact, our techniques still apply, but without 
a bound of the form (\ref{bound_from_below}) and 
further assumptions on the initial data, the size of 
the time interval where the solution exists can tend to zero near the asymptotic region of 
$\Si$. We therefore assume (\ref{bound_from_below}),
briefly commenting on more general situations in section \ref{further}.
Also, for $\Si$ non-compact, in order to accommodate conditions at infinity which 
are not too restrictive,
we shall employ the uniformly local Sobolev spaces\footnote{Other function spaces can be
used. In particular,  the ``little $\ell_p$-Sobolev spaces", $\ell_p(H^s)$, could be employed, 
these 
being, in fact, better suited for treating more general data on non-compact manifolds, see e.g.
\cite{TriebelBook}. }
 $H^s_{ul}$ originally introduced by 
Kato \cite{KatoQL}. Their definition and basic properties are recalled in the 
appendix \ref{ul_Sobolev_appendix}; although it comes as
no surprise that $H^s_{ul}$ and $H^s$ are equivalent (as Banach spaces) when $\Si$ is compact.
It is assumed that the metric employed to define the spaces $H^s_{ul}(\Si)$
is equivalent, up to the sign convention, to the metric $g_0$.

Yet another difficulty which arises in the non-compact case is that $u$ can become
arbitrarily close to the boundary of the light-cone in the asymptotic region. This would lead
to a breakdown of the positivity of certain matrices required for our proofs, again
implying that the time interval on which the solutions exist cannot be made uniform.
To prevent this, we impose a bound on the size of the initial tangent velocity $v$. To see
that 
such a bound gives the desired control over $u$ near the initial hypersurface, simply notice that
in normal coordinates at a point $p$ on the hypersurface with $\frac{\partial}{\partial x^0}$ 
normal 
to $\Si$, the set of unit length future-directed
vectors on $T_p \cM$ is given as usual by the upper sheet of the time-like hyperboloid 
which is asymptotic to the light-cone.

We are now in a position to state our main result.

\begin{theorem}
Let $(\Si, g_0, \kappa, r_0, \scr_0, v, \ccP)$ be an initial data set for the 
Einstein-Euler-Entropy system, with 
$g_0 \in H_{ul}^{s+1}(\Si)$,
 $\kappa$, $v$,
 $r_0$, $\scr_0\in H_{ul}^{s}(\Si)$, where 
 $s > \frac{3}{2} + 2$.
Suppose further that  $\scr_0 \geq 0$,
 $r_0 \geq c_1$, $\frac{p_0 + \varrho_0}{r_0} \geq c_1$,
$\nu_0^2 = \nu^2(r_0,\scr_0) \geq c_1$, 
and $|v|_{g_0} \geq -c_2$, for some constants $c_1,\, c_2>0$.
 Then there 
exists an Einsteinian 
development $(\cM,g)$  of $(\Si, g_0, \kappa, r_0, \scr_0, v, \ccP)$
which is a perfect fluid source, with 
$\cM$ diffeomorphic to $[0,T_E] \times \Si$ for some real number $T_E>0$.
Moreover,
$g \in C^0([0,T_E], H_{ul}^{s+1}(\Si)) \cap 
C^1([0,T_E], H_{ul}^{s}(\Si)) \cap C^2([0,T_E], H_{ul}^{s-1}(\Si))$, the 
matter density $\varrho$ and the pressure $p$ of the fluid source are 
given by (\ref{eq_of_state}) and
(\ref{pressure}), respectively, where the functions $r$ and $s$ in 
(\ref{eq_of_state})-(\ref{pressure}) belong to 
$C^0([0,T_E], H_{ul}^{s}(\Si)) \cap C^1([0,T_E], H_{ul}^{s-1}(\Si))$, satisfy $r > 0$, 
$s \geq 0$, and are such that
$\left. r \right|_{\Si} = r_0$,  
$\left. s \right|_{\Si} = \scr_0$, and $\nu^2 = \nu^2(r, s) > 0$. Furthermore, denoting by $u$ the unit time-like vector field of 
the fluid source, we have that $u \in C^0([0,T_E], H_{ul}^{s}(\Si)) \cap C^1([0,T_E], H_{ul}^{s-1}(\Si))$
 and $\pi_g(u) = v$, where
$\pi_g:  \left. T\cM \right|_\Si \rar T\Si$ is the orthogonal projection onto the
tangent bundle of $\Si$.
\label{main_theorem}
\end{theorem}

\begin{rema}
Notice that the pointwise inequalities of theorem \ref{main_theorem} all make sense in that
$s > \frac{3}{2} + 2$ implies that the quantities involved are continuous. 
As discussed above, when $\Si$ is compact, the hypotheses involving $c_1$ can
be relaxed by assuming that those quantities are simply greater than zero, $H^s_{ul}$ can be replaced
by the ordinary Sobolev spaces and the bound on $|v|_{g_0}$ is automatically satisfied
(recall that $g_0$ is negative definite).
\end{rema}

\section{The Einstein-Euler-Entropy system in the frame formalism\label{frame_formalism}.}
In this section, we shall use the so-called frame formalism to 
write a different set of equations for the Einstein-Euler-Entropy system.
It will be shown in section \ref{proof_main_theorem_section} that solutions
to the new system imply existence of solutions to the original 
Einstein-Euler-Entropy equations. 
As in most analytic investigations of the Einstein equations, the point of 
view is essentially local as a consequence of the phenomenon 
of finite propagation speed. Thus a chart $U$ should be implicitly understood whenever
coordinates are involved. If $\Si$ is a space-like three surface, or a \Em{slice} for short, 
then we shall use a slight abuse of notation and still denote by $\Si$ the set
$\Si \cap U$.

In the frame formalism, the information about the metric is encoded in an orthonormal
frame $\{ e_\mu \}_{\mu=0}^3$, where the frame is related to the coordinate basis by
\begin{gather}
 e_\mu = e^A_{\ms \mu} \frac{\partial}{\partial x^A}.
\label{frame_coefficients}
\end{gather}
\begin{notation}
From now on, unless otherwise specified, all tensor fields will be expressed in the 
orthonormal frame $\{ e_\mu \}_{\mu=0}^3$, with Greek letters used to indicate the components
of such fields in the basis $\{ e_\mu \}_{\mu=0}^3$. A bar over an index, e.g., $\bar{\al}$, indicates
that it can take only the values $1,2$ or $3$, with summation of barred indices being only over $1,2,3$
as well. The only exception for our choice of basis will be for the frame itself, which
will be given in terms of the frame coefficients $e^A_{\ms \mu}$ in (\ref{frame_coefficients}).
As  in (\ref{frame_coefficients}), capital Latin letters range from $0$ to $3$ and will be used 
to denote components with respect to 
the coordinate basis; a bar, e.g., $\bar{A}$, indicates restriction to $1, 2$ or $3$.
Indices are still raised and lowered with the metric $g$, as usual.
\label{notation_barred_indices}
\end{notation}

By construction, in terms of the orthonormal
frame $\{ e_\mu \}_{\mu=0}^3$, the metric $g$ is always represented by the Minkovski metric
$g_{\al\be} = \operatorname{diag}(1, -1, -1, -1 )$. The relation to the metric 
in the basis $\{ \frac{\partial}{\partial x^A} \}_{A=0}^3$ is given by
\begin{gather}
 g^{AB} = e^A_{\ms \al} e^B_{\ms \be} g^{\al\be}.
\label{def_metric_frame_coeff}
\end{gather}
The 
\Em{connection coefficients} $\Ga_{\al \ms \be}^{\ms \ga} $ are defined via
\begin{gather}
 \nabla_\al e_\be = \Ga_{\al \ms \be}^{\ms \ga} e_\ga.
 \label{conn_coeff}
\end{gather}
The condition that $\nabla$ is compatible with the metric then takes the form
\begin{gather}
 \Ga_{\al \ms \be}^{\ms \mu} g_{\mu \ga} + \Ga_{\al \ms \ga}^{\ms \mu} g_{\mu \be} = 0.
\label{metric_compatible}
\end{gather}
One of the key ingredients of the formalism we shall employ is to treat
the connection coefficients as unknowns in their own right.
As $\Ga_{\al \ms \be}^{\ms \ga}$ is given in terms of first derivatives of the metric,
treating them as independent variables allows us to express the original 
Einstein-Euler-Entropy system, which involves second derivatives of $g$, as a first
order system\footnote{This argument  only presents the basic heuristic intuition of 
the method. In fact, although the metric will not be one of the basic unknowns,
our system would be third order in the metric if written in terms of it. This is 
because the system will involve first derivatives of the Weyl tensor, see equations (\ref{Friedrichs_tensor}),
(\ref{div_Friedrichs_tensor}) and (\ref{div_Friedrichs_tensor_eq}). It should be noticed that 
this is not an oddity of our 
formalism. If one tries to approach the 
Einstein-Euler-Entropy system in the usual formalism of second order equations, 
it cannot 
be directly solved as it stands; a 
quasi-diagonalization process has then to be carried out, leading 
to a system that is also third order in the metric \cite{C2}. }.

Symmetry (\ref{metric_compatible}) will be assumed throughout. In other words,
we shall only need evolution equations for $24$  independent 
$\Ga$'s\footnote{Due to (\ref{metric_compatible}), there are $\frac{n^2(n-1)}{2}$ independent 
components of $\Ga_{\al \ms \be}^{\ms \ga}$ in $n$ 
space-time dimensions, $n=4$ in our case.}
with the remaining
components explicitly defined via (\ref{metric_compatible}).

Since it is not possible to decide from the connection coefficients alone whether $\nabla$ is the
Levi-Civita connection of $g$ \cite{Schmidt}, further conditions will be necessary. The requirement
that the connection is torsion-free, along with the (once contracted) Bianchi identities and 
the standard decomposition of the Riemann curvature tensor in terms of the Weyl and Schouten 
tensors, will be imposed as further equations of motion of the system. In this regard,
we introduce the tensor $d^\al_{\ms\be\ga\de}$ ($d$ for ``decomposition'') defined as
\begin{gather}
 d^\al_{\ms\be\ga\de} := R^\al_{\ms\be\ga\de} - W^\al_{\ms \be \ga \de} - 
g^\al_{\ms [\ga} S_{\de]\be}^{} + g_{\be[\ga}^{}S_{\de]}^{\ms\al},
\label{decomposition}
\end{gather}
where $[\al\,\be]$ means that the indices are anti-symmetrized;
$W^\al_{\ms \be\ga\de}$ is the Weyl tensor; $S_{\al\be}$ is the Schouten tensor,
given by
\begin{gather}
 S_{\al\be} := R_{\al\be} - \frac{1}{6} R g_{\al\be},
\label{Schouten}
\end{gather}
and $R^\al_{\ms\be\ga\de}$ is the Riemann tensor, which can be written as
\begin{align}
\begin{split}
 R^\al_{\ms \be\ga\de} = & \,  e_\ga(\Ga_{\de \ms \be}^{\mss \al}) - 
e_\de(\Ga_{\ga \ms \be}^{\mss \al}) - \Ga_{\mu \ms \be}^{\mss \al}(\Ga_{\ga \ms \de}^{\mss \mu} - \Ga_{\de \ms \ga}^{\mss \mu})
+ \Ga_{\ga \ms \mu}^{\mss \al} \Ga_{\de \ms \be}^{\mss \mu} \\
& - \Ga_{\de \ms \mu}^{\mss \al} \Ga_{\ga \ms \be}^{\mss \mu}.
\end{split}
\label{definition_Riemann}
\end{align}
By construction, $d^{\al}_{\ms\be\ga\de}$ possesses the usual 
symmetries of the Riemann tensor, since such symmetries are shared by
$W^\al_{\ms \be \ga \de}$ 
and 
$-g^\al_{\ms [\ga} S_{\de]\be}^{} + g_{\be[\ga}^{}S_{\de]}^{\ms\al}$.
Recall that the \Em{torsion} of the connection is the tensor $\cT$ defined via
\begin{gather}
 \cT_{\al \msm \be}^{\mss \mu} e_\mu = -[e_\al,e_\be] + (\Ga_{\al \ms \be}^{\mss \mu} - \Ga_{\be \ms \al}^{\mss \mu} )e_\mu,
\nonumber
\end{gather}
where $[\cdot,\cdot]$ is the usual commutator of two vector fields. Notice that 
\begin{gather}
 \cT_{\al \msm \be}^{\mss \mu} = - \cT_{\be \msm \al}^{\mss \mu}.
\label{torsion}
\end{gather}
Aiming at the Bianchi identities, we define the \Em{Friedrich tensor} by
\begin{gather}
 F^{\al}_{\ms \be\ga\de} := W^\al_{\ms \be \ga \de} - 
g^\al_{\ms [\ga} S_{\de]\be}^{},
\label{Friedrichs_tensor}
\end{gather}
and let
\begin{gather}
 F_{\al\be\ga} := \nabla_\mu F^\mu_{\ms \al \be \ga}.
\label{div_Friedrichs_tensor}
\end{gather}
To understand the role of $F_{\al\be\ga}$, consider for simplicity the case of vacuum.
Then the once-contracted Bianchi identity
\begin{gather}
\nabla_\mu R^{\mu}_{\mss \al\be\ga} = \nabla_\be R_{\al \ga} - \nabla_\ga R_{\al \be},
\nonumber
\end{gather}
combined with Einstein equations (\ref{Einstein_eq}) yields $F_{\al\be\ga} = 0$ whenever
the decomposition of the Riemann tensor holds, i.e., $d^\al_{\mss\be\ga\de} = 0$.

The last definition we need to introduce is 
\begin{gather}
q_\al :=  (p + \varrho) u^\mu \nabla_\mu u_\al - \nu^2 (p + \varrho) u_\al \nabla_\mu u^\mu  - \nabla_\al p.
\label{q_al}
\end{gather}
$q_\al$ will replace  (\ref{Euler_2}) in the new set of equations. To motivate 
this, suppose
that we have a solution to the Einstein-Euler-Entropy system. Then
differentiating (\ref{eq_of_state}) in the direction of $u$, using (\ref{pressure}),  
(\ref{rest_mass_conservation}),  (\ref{loc_adiabatic}), and (\ref{sound_speed}) yield
\begin{gather}
u^\al \nabla_\al p = -(p + \varrho) \nu^2 \nabla_\al u^\al.
\nonumber
\end{gather}
Using this into (\ref{Euler_2}) then gives $q_\al = 0$.

Tracing Einstein equations
(\ref{Einstein_eq}) gives
\begin{gather}
 R = -\K T,
\label{trace_Einstein_eq}
\end{gather}
where $T$ is the trace of the stress-energy tensor. Using 
(\ref{Einstein_eq}) and (\ref{trace_Einstein_eq}) into (\ref{Schouten}) produces
\begin{gather}
 S_{\al\be} = \K( T_{\al\be} - \frac{1}{3} T g_{\al\be}),
\label{Einstein_eq_equivalent}
\end{gather}
which is an equivalent way of writing the Einstein equations.

We can now define the \Em{Einstein-Euler-Entropy system in the frame formalism}
as the system comprised of
\begin{subnumcases}{\label{EEE_frame_system}}
\text{equations (\ref{Euler_1}), (\ref{rest_mass_conservation}), (\ref{eq_of_state}), (\ref{pressure}), 
(\ref{loc_adiabatic}), (\ref{sound_speed}),
(\ref{Einstein_eq_equivalent})}, &  \\
q_\al = 0, & \label{q_al_equation} \\
\cT_{\al \msm \be}^{\mss \mu} = 0, & \label{no_torsion_eq} \\
 d^\al_{\ms\be\ga\de} = 0, & \label{decomposition_eq} \\
F_{\al\be\ga} = 0, & \label{div_Friedrichs_tensor_eq} 
\end{subnumcases}
subject to the constraint
$u^\al u_\al = 1$, with $T_{\al\be}$ given by (\ref{perfect_fluid_source}), 
and where $q_\al$, $\cT_{\al \msm \be}^{\mss \mu}$, 
$d^\al_{\ms\be\ga\de}$ and $F_{\al\be\ga}$ are given
by (\ref{q_al}), (\ref{torsion}), (\ref{decomposition}),
and (\ref{div_Friedrichs_tensor}), respectively.
The unknowns to be determined are the frame coefficients $e^A_{\ms \al}$, 
the connection coefficients 
$\Ga_{\al \ms \be}^{\ms \ga}$, 
the Weyl tensor $W^\al_{\ms \be\ga\de}$, the four-velocity $u$, 
the rest mass density $r$, the specific entropy $s$, and
the matter density $\varrho$. 
The usual symmetries of $W^\al_{\ms \be\ga\de}$, and those of 
$\Ga_{\al \ms \be}^{\ms \ga}$ determined by (\ref{metric_compatible}),
are explicitly assumed,  therefore the system is written for only the $20$ independent components 
of the Weyl tensor\footnote{In $n$ space-time dimensions, 
$W^\al_{\ms \be\ga\de}$ has $\frac{1}{12}n(n+1)(n+2)(n-3)$ independent components.
As explained in remark \ref{remark_trace_W}, the trace-free condition is not assumed in the system, thus there is
an extra freedom of $\frac{1}{2}n(n+1)$  components.}
 and the $24$ independent connection coefficients\footnote{All the symmetries and which components enter in the
 system are described in section \ref{reduced_section}, where we also
 write the system in a more explicit form.}.
 
 \begin{rema}
 Obviously, when attempting to solve (\ref{EEE_frame_system}), it is not yet
 known that $W$ is the Weyl tensor. In particular, the trace-free condition must
  be demonstrated.
\label{remark_trace_W}
\end{rema}

One way to motivate the choice of the Weyl tensor as one of the unknowns is to  
consider the 
simpler case of vacuum. Then, the 
validity of $d^\al_{\ms\be\ga\de} = 0$ gives 
$R^\al_{\ms \be\ga\de} = W^\al_{\ms \be\ga\de}$ which, upon contraction, produces
the vacuum Einstein equations. Since the Bianchi identities are necessary conditions for 
the solvability of Einstein equations, we also have
$F_{\be\ga\de} = \nabla_\al W^\al_{\ms \be\ga\de} 
= \nabla_\al R^\al_{\ms \be\ga\de} = 0$ in the case of vacuum.
We see in this way that the usual vacuum Einstein equations can be recovered from 
the Weyl tensor\footnote{Other choices of variables are, of course, possible. Choquet-Bruhat and
York derived a different system --- also based on the Bianchi identities --- where the Riemann tensor is 
one of the unknowns of the problem \cite{ChoquetYork}.}.

Equation (\ref{div_Friedrichs_tensor_eq}) is sometimes referred to as the ``Bianchi equation",
and its importance in General Relativity has been long recognized. 
It figures in the Newman-Penrose's spin formalism and has been used in the study
of massless fields \cite{Penrose_zero}, including gravitational radiation \cite{Sachs}.
Friedrich has employed it extensively, initially to obtain energy estimates 
in terms of the Bel-Robinson tensor \cite{Friedrich1, Friedrich2, Friedrich3}, and later to
derive several hyperbolic reductions for the Einstein equations, both in 
the vacuum case \cite{FriHypGauge, FriHypRed} and coupled to matter
\cite{Fri, FriRen}, with some of these results extended in \cite{MagHyd}. 
Friedrich and Nagy also used (\ref{div_Friedrichs_tensor_eq}) in their study of the 
initial-\emph{boundary} value problem in General Relativity \cite{FriNagy}.
Using a different point of view 
than Friedrich, the Bianchi equation was used to derive first order symmetric hyperbolic reduced Einstein equations that 
treat the Riemann tensor as one of the unknowns 
by Anderson, Choquet-Bruhat and
York, both in the 
vacuum and matter 
cases\footnote{By the time of the completion of this manuscript, we learned
of \cite{Vamsi}, where the formulation of 
Anderson, Choquet-Bruhat and
York is used to give yet another proof of short-time 
existence for the Einstein-Euler-Entropy system.}
\cite{ChoquetYorkAnderson,ChoquetYork}.
Finally, equation (\ref{div_Friedrichs_tensor_eq}) has also been employed (in a rather elaborated
fashion though) by Christodoulou and Klainerman in their impressive proof of the global 
non-linear stability of Minkovski space.

\subsection{Gauge fixing.}

The system (\ref{EEE_frame_system}) is overdetermined. In order to obtain a reduced system, which 
is determined and hyperbolic in a precise sense, a  gauge choice has to be made.
A specific decomposition of the Weyl tensor, suitable for our gauge choice, will also be necessary.

Put 
\begin{gather}
\pi_{\al\be} := g_{\al\be} - u_\al u_\be,
\nonumber
\end{gather}
so $\pi_{\al\be}$ is the metric induced on the space orthogonal to $u$, with projection given
by  
\begin{gather}
\pi_\al^{\mss \be} = g_\al^{\mss \be} - u_\al u^\be 
\nonumber
\end{gather}
(recall that indices are always raised with $g_{\al\be}$).

\begin{rema}
Despite the terminology, in general, $\pi_{\al\be}$ and $\pi_\al^{\mss \be}$
will not agree with the metric and projection on the $t=\,$constant space-like slices,
because such slices are not expected to be orthogonal to $u$ except in some special cases. 
In particular, 
$\left. \pi_{\al\be} \right|_\Si$ does not, in general, agree with $g_0$.
\end{rema}

Letting $\ve_{\al\be\ga\de}$ be the totally anti-symmetric tensor, with the usual
convention $\ve_{0123} = +1$, we define 
\begin{gather}
\ve_{\al\be\ga} := \ve_{\mu \nu \si \tau} u^\mu \pi_\al^{\mss \nu} \pi_\be^{\mss \si}
\pi_\ga^{\mss \tau}.
\nonumber
\end{gather}
Expanding out, it is easy to obtain
\begin{gather}
\ve^{\mu \al \be } \ve_{\mu \ga \de} = -2 \pi^{\al}_{\ms [\ga } \pi^{\be}_{\mss\de ]},
\label{identity_ve_1}
\end{gather}
and
\begin{gather}
\ve^{\mu \nu \al } \ve_{\mu \nu \be } = -2 \pi^{\al}_{\mss \be}.
\label{identity_ve_2}
\end{gather}

Recall that an observer moving with a (time-like) $4$-velocity $V$ measures an electric 
and a magnetic field  given by $\cF_{\al \be} V^\be$ and $\cF_{\al\be}^* V^\be$, respectively, 
where $\cF$ is the Maxwell stress tensor and $\cF^*$ its dual. By analogy, it is 
customary to introduce the following.

\begin{defi}
Let $W_{\al \be\ga\de}$ be the Weyl tensor and 
$W^*_{\al \be\ga\de}$ its dual, given by
\begin{gather}
W^*_{\al \be\ga\de} = \frac{1}{2} \ve_{\ga\de}^{\mss \mss \mu \nu}  W_{\al \be \mu \nu}^{}.
\nonumber
\end{gather}
The $u$-\Em{electric} and $u$-\Em{magnetic parts} of $W^\al_{\ms \be\ga\de}$
are defined by
\begin{gather}
E_{\al\be} : = W_{\mu \nu \si \tau} u^\mu u^\si \pi_\al^{\mss \nu} \pi_\be^{\mss \tau},
\nonumber
\end{gather}
and
\begin{gather}
B_{\al\be} : = W^*_{\mu \nu \si \tau} u^\mu u^\si \pi_\al^{\mss \nu} \pi_\be^{\mss \tau},
\nonumber
\end{gather}
respectively.
\label{u_elec_u_mag_def}
\end{defi}
It follows that $E_{\al\be}$ and $B_{\al\be}$ are symmetric and trace-free. 
With the help of (\ref{identity_ve_1}) and (\ref{identity_ve_2}), it is not difficult 
to verify the following.

\begin{lemma}
The following decompositions of the Weyl tensor and its dual hold: 
\begin{align}
\begin{split}
W_{\al \be\ga\de}  & =   2 \Big ( 
\pi_{\be [\ga } E_{\de ] \al} 
- u_\be u_{ [\ga } E_{\de ] \al} 
- \pi_{\al [\ga} E_{\de] \be} 
+  u_\al u_{ [\ga} E_{\de] \be}
 \Big ) \\
& - 2 \Big( u_{[\ga} B_{\de] \mu} \ve^{\mu}_{\mss \al \be} + 
u_{[\al } B_{\be]\mu} \ve^{\mu}_{\mss \ga \de} \Big ),
\end{split}
\label{decomp_W}
\end{align}
and
\begin{align}
\begin{split}
W^*_{\al \be \ga \de} =& \,  2 u_{ [ \al } E_{\be ] \mu} \ve^{\mu}_{\mss \ga \de}
- 4 E_{\mu [\al}^{} \ve_{\be] \ms  [\ga }^{\msm \mu } u_{\de ]}^{} 
- 4 u_{ [ \al } B_{\be ] [ \ga } u_{\de]} \\
& - B_{\mu \nu} \ve^{\mu}_{\mss \al \be} \ve^{\nu}_{\mss \ga \de}.
\end{split}
\label{decomp_W_dual}
\end{align}
\label{decomp_W_lemma}
\end{lemma}
Decompositions (\ref{decomp_W}) and (\ref{decomp_W_dual}) allow the Weyl tensor to be eliminated
from the system (\ref{EEE_frame_system}) in favor of its electric and magnetic components, producing 
yet another set of equations where
$E_{\al\be}$ and $B_{\al\be}$ will be unknowns; see section \ref{reduced_section}.

Until now, $\{ e_\mu \}_{\mu=0}^3$ has been an arbitrary frame with respect to which
all tensor fields have been written. In particular, there is no relation so far between  
$u$ and the frame other than the general fact that $u$ can be decomposed in this base, i.e., 
$u = u^\mu e_\mu$. A \Em{gauge choice} in our formalism will be a specific choice of frame --- very 
much in the same way that a choice of gauge in the coordinate formalism corresponds to 
a determined choice of coordinates, e.g., wave coordinates. 

\begin{defi}
Let $(\cU,g,u)$ be a fluid source. An orthonormal frame $\{ e_\mu \}_{\mu=0}^3$ in $\cU$ is 
called a \Em{fluid source gauge} if it satisfies
\begin{gather}
e_0 = u, 
\label{gauge_e_0}
\end{gather}
with the remaining $\{ e_\mu \}_{\mu = 1}^3$ being 
\Em{Fermi propagated} along $e_0$. By definition, this means
\begin{gather}
\langle \nabla_{e_0} e_{\bar{\al}}, e_{\bar{\be}} \rangle = 0,
\label{Fermi_transport}
\end{gather}
for $\bar{\al},\bar{\be} = 1,2,3$ (recall our 
conventions in notation \ref{notation_barred_indices}).
\label{def_fluid_source_gauge}
\end{defi}
Condition (\ref{Fermi_transport}) means that $\{ e_{\bar{\al}} \}$ remains orthogonal along the 
time-flow of $u$. Roughly speaking, this can be understood as similar to parallel transport, 
with the important difference that $\{ e_{\bar{\al}} \}$ does not remain parallel but is allowed
to ``rotate about time-axis given by $u$."

Let $t$ be a parameter for the flow lines of $u$ and $\Si$ be a slice. Then
there exists a foliation of a neighborhood of $\Si$ by 
leaves $\Si_t = \{ t = \text{ constant} \}$,
which are space-like and diffeomorphic to $\Si$. Let  
$\{ \fracbarAline  \}_{\bar{A}=1}^3$
be tangent vectors on $\Si$ associated with coordinates $\{ x^{\bar{A}} \}_{\bar{A}=1}^3$. The
coordinates $\{ x^{\bar{A}} \}_{\bar{A}=1}^3$
can be dragged along $u$ to give coordinates on $\Si_t$, and since $u$ is time-like, 
$u \not\in \operatorname{span}\{ \fracbarAline  \}$. Setting 
$ \frac{ \partial }{\partial x^0} := \frac{\partial }{\partial t} \equiv u$ then 
gives a basis 
$\{ \frac{ \partial }{\partial x^0}, \fracbarAline  \}_{\bar{A}=1}^3
 \equiv \{  \frac{\partial}{\partial x^A }\}_{A=0}^3$
for the space-time tangent space, with coordinates $\{ x^A \}_{A=0}^3$. From now on, 
it will be assumed that whenever a fluid source gauge is employed, the coordinates are
arranged as just described, unless stated otherwise.

The following is a simple consequence of our choice of gauge and
the fact that $g$ is represented by the Minkovski metric in 
the frame formalism.

\begin{lemma}
In fluid source gauge, it holds that 
\begin{gather}
e^A_{\ms 0} = \de^A_{\ms 0}, \, u^\al = \de^\al_{\mss 0} , \, 
\Ga_{0 \ms \bar{\be}}^{\mss \bar{\al}} = 0, 
\nonumber \\ 
\pi_{\al\be} = g_{\al\be} - \de_{0\al}\de_{0\be}, 
\,
\pi_{\bar{\al}\bar{\be}} = -\de_{\bar{\al}\bar{\be}},
\, 
\pi_{\al}^{\mss \be} = \de_{\al}^{\ms\be} - \de_{0\al} \de^{0\be}.
\nonumber
\end{gather}
\label{gauged_quantities_lemma}
\end{lemma}

\subsection{The reduced system of equations.\label{reduced_section}}
In this section, we investigate a reduced system for equations 
(\ref{EEE_frame_system}). As mentioned in the introduction, its derivation, which 
follows that of 
\cite{FriHypRed, Fri, FriRen}, is given in the appendix \ref{derivation_appendix}.

In light of the various symmetries 
involved, only equations for some components of the tensors involved are needed. 
First let us  determine them.

Taking the inner products of (\ref{conn_coeff}) with $e_\de$ and using
(\ref{metric_compatible}) gives
\begin{gather}
\left. \Ga_{\al \ms \be}^{\ms \be} \right|_{ \nosum } = 0,
\label{gamma_2_indices}
\end{gather}
\begin{gather}
\Ga_{\bar{\al}\ms \bar{\be}}^{\ms 0} = \Ga_{\bar{\al} \ms 0}^{\ms \bar{\be}},
\label{gamma_sym_two_bar}
\end{gather}
\begin{gather}
\Ga_{\bar{\al}\ms\bar{\be}}^{\ms\bar{\ga} } = -\Ga_{\bar{\al}\ms\bar{\ga}}^{\ms\bar{\be} },
\label{gamma_sym_spatial}
\end{gather}
where  $\left. (\cdot) \right|_{\nosum}$ indicates that there is no sum over repeated
indices, and 
from (\ref{metric_compatible}) and (\ref{gamma_2_indices}) we have
\begin{gather}
\Ga_{0\ms0}^{\mss\al} = \Ga_{0 \ms \al}^{\mss 0}.
\label{gamma_sym_two_0}
\end{gather}
Identities (\ref{gamma_2_indices}), (\ref{gamma_sym_two_bar}), 
(\ref{gamma_sym_spatial}) and 
(\ref{gamma_sym_two_0}) hold for any frame. Fixing the gauge allows further simplifications.
From now on, we shall assume that our frame is a fluid source gauge, unless stated
otherwise.

From lemmas \ref{decomp_W_lemma}, \ref{gauged_quantities_lemma} and the 
symmetries of the Weyl tensor,
\begin{gather}
E_{\al 0} \, (=E_{0\al} ) = 0 , \,  B_{\al 0} \, (=B_{0\al}) = 0, 
\label{E_0_B_0_gauge} \\ 
E_{\bar{\al}\bar{\be}} = W_{0\bar{\al}0\bar{\be}}, \,
B_{\bar{\al}\bar{\be}} = \frac{1}{2} 
W^{}_{0\bar{\al}\mu \nu}\ve_{0\bar{\be}}^{\mss\mss \mu \nu}.
\label{E_bar_B_bar_gauge}
\end{gather}
From (\ref{E_0_B_0_gauge}) and lemmas \ref{decomp_W_lemma} and
\ref{gauged_quantities_lemma},
 it is seen that in fluid source gauge the Weyl tensor is determined from
 $E_{\bar{\al}\bar{\be}}$ and $B_{\bar{\al}\bar{\be}}$. It is therefore
 enough to consider equations for these quantities, whereas from
 lemma \ref{gauged_quantities_lemma}  and identities 
(\ref{gamma_2_indices}), (\ref{gamma_sym_two_bar}), 
(\ref{gamma_sym_spatial}) and 
(\ref{gamma_sym_two_0}), we obtain that it suffices to have
evolution equations for connection coefficients 
$\Ga_{\bar{\al}\ms \bar{\ga} }^{\ms \bar{\be} }$, $\Ga_{0 \ms \bar{\al}}^{\mss 0}$,
$\Ga_{\bar{\al}\ms \bar{\be} }^{\ms 0}$ and frame coefficients 
$e^A_{\ms \bar{\al}}$.

To write the reduced system, the introduction of yet another variable is needed. Write
\begin{gather}
s_\al =\nabla_\al s.
\label{s_al_def}
\end{gather}
As was done for the quantities $e^A_{\ms \beta}$, 
$\Ga_{\al \msm \be}^{\ms \ga}$, $E_{\al\be}$, $B_{\al\be}$, we shall treat
$s_\al$ as an unknown, with the relation (\ref{s_al_def}) to be demonstrated after solutions
are obtained.

Define the operator $\D_\mu$ which acts on fields of 
the orthogonal complement of $e_0$ by
\begin{gather}
 \D_\mu A_{\bar{\al}_1 \bar{\al_2} \cdots \bar{\al}_\ell }
= \pi_\mu^{\mss \nu} \nabla_\nu 
A_{\be_1 \be_2 \cdots \be_\ell }
\pi_{\bar{\al}_1}^{\msb \be_1} \pi_{\bar{\al}_2}^{\msb \be_2} \cdots \pi_{\bar{\al}_\ell}^{\msb \be_\ell}.
\label{spatial_derivative}
\end{gather}
Then
\begin{gather}
\D_\mu \pi_{\bar{\al}\bar{\be}} = 0 \, \text{ and } \, 
\D_\mu \ve_{\bar{\al}\bar{ \be}\bar{ \ga} } = 0.
\nonumber
\end{gather}
Although, in fluid source gauge, it then follows that 
\begin{gather}
 \D_{\bar{\mu}} A_{\bar{\al}_1 \bar{\al_2} \cdots \bar{\al}_\ell }
=  \nabla_{\bar{\mu}} A_{\bar{\al}_1 \bar{\al_2} \cdots \bar{\al}_\ell },
 \nonumber
\end{gather}
we shall explicitly write $\D_{\bar{\mu}}$ to facilitate 
the comparison with appendix 
\ref{derivation_appendix}.

We can now investigate the reduced system, whose equations are
\begin{subequations}{\label{reduced_system}}
\begin{align}
& \partial_t e^A_{\ms \bar{\beta}} + (\Ga_{\bar{\be}\msm 0}^{\ms \bar{\mu}}  - 
 \Ga_{ 0\ms \bar{\be} }^{\ms \bar{\mu} } )e^A_{\ms \bar{\mu}} 
   -  \Ga_{0 \ms \bar{\beta}}^{\ms 0 }  \de^A_{\ms 0} = 0  
 \label{reduced_eq_frame}
  \\ 
 & \partial_t \Ga_{\bar{\de}\msm \bar{\beta}}^{\ms \bar{\al}} + 
 \Ga_{\bar{\la}\ms \bar{\beta}}^{\ms \bar{\al}}\Ga_{\bar{\de}\msm 0}^{\ms \bar{\la}} 
 + \Ga_{0\ms 0}^{\mss \bar{\al}} \Ga_{\bar{\de} \ms \bar{\be}}^{\mss 0}  - \Ga_{\bar{\de}\msm 0}^{\mss \bar{\al}} \Ga_{0\msm \bar{\be}}^{\ms 0} - \ve^{\mu \bar{\al}}_{\msb\bar{\be}} B_{\bar{\de}\mu} = 0
 \label{reduced_eq_Ga_bar}
  \\ 
& \partial_t \Ga_{0 \msm \bar{\al}}^{\ms 0} - \nu^2 e_{\bar{\la}}(\Ga_{\bar{\al}\msm 0}^{\ms \bar{\la}})
+ \Ga_{0\msm \bar{\mu}}^{\ms 0} \Ga_{\bar{\al} \msm 0}^{\ms \bar{\mu}}
-\frac{\partial \nu^2}{\partial s}  \Ga_{\bar{\mu}\msm 0}^{\ms \bar{\mu}} s_{\bar{\al}}
+ \frac{1}{p+\varrho} 
 \nonumber \\
& \hspace{0.5cm}
\left(1 + \frac{r}{\nu^2} \frac{\partial \nu^2}{\partial r} \right)\frac{\partial p}{\partial s}
 \Ga_{\bar{\mu}\msm 0}^{\ms \bar{\mu}} s_{\bar{\al}} 
 +  \left(\frac{p+\varrho}{\nu^2}\left(\frac{\partial^2 p}{\partial \varrho^2}\right)_{\hspace{-0.2cm} s}
   - \nu^2 \right) \Ga_{\bar{\mu}\msm 0}^{\ms \bar{\mu}}\Ga_{0 \msm \bar{\al}}^{\ms 0}
 \nonumber \\
& \hspace{0.5cm}
- \nu^2 \Big [ \Ga_{\mu \msm 0}^{\ms \bar{\la}}(\Ga_{\bar{\al} \ms \bar{\la}}^{\ms \mu} - 
\Ga_{\bar{\la} \ms \bar{\al}}^{\ms \mu}) - \Ga_{\bar{\al} \ms \mu}^{\ms \bar{\la}} 
\Ga_{\bar{\la} \msm 0}^{\ms \mu} 
 - \Ga_{\bar{\la} \ms \mu}^{\ms \bar{\la}} \Ga_{\bar{\al} \msm 0}^{\ms \mu} 
\Big ]
\nonumber \\
& \hspace{0.5cm}
- 
\frac{1}{p+\varrho} \nu^2 \frac{\partial \varrho}{\partial s}\Ga_{\bar{\mu}\msm 0}^{\ms \bar{\mu}} s_{\bar{\al}} 
 = 0 
\label{reduced_eq_Ga_0_0}
\\ 
& \nu^2 \partial_t \Ga_{\bar{\al}\ms \bar{\be}}^{\ms 0} 
- \nu^2 e_{\bar{\be}}(\Ga_{0\msm\bar{\al}}^{\ms0} )
- \nu^2 E_{\bar{\al}\bar{\be}}
  -\nu^2 \Big [ \Ga_{\mu \ms \bar{\be} }^{\ms 0}( \Ga_{0 \msm \bar{\al}}^{\ms \mu} -
  \Ga_{\bar{\al} \msm 0}^{\ms \bar{\mu} } )
 \nonumber \\
&  \hspace{0.5cm}
    - \Ga_{0 \msm \bar{\mu}}^{\ms 0} \Ga_{\bar{\al}\ms \bar{\beta}}^{\ms \bar{\mu}} 
  + \Ga_{0 \msm \bar{\mu}}^{\ms 0}( \Ga_{\bar{\al}\ms \bar{\be}}^{\ms\bar{\mu}} -
  \Ga_{\bar{\be}\ms \bar{\al}}^{\ms\bar{\mu}} )
     + 
  \nu^2 \Ga_{\bar{\mu}\msm 0}^{\ms\bar{\mu}}( \Ga_{\bar{\al}\ms \bar{\be}}^{\ms 0} -
  \Ga_{\bar{\be}\ms \bar{\al}}^{\ms 0} )
   \nonumber \\
&   \hspace{0.5cm}  + \frac{1}{p + \varrho}\left( \frac{\partial \varrho}{\partial s} - 
     \frac{1}{\nu^2}\frac{\partial p }{\partial s} \right)
    (\Ga_{0\ms \bar{\al}}^{\ms 0} s_{\bar{\be}} - \Ga_{0 \ms \bar{\be}}^{\ms 0} 
    s_{\bar{\al}} ) 
 \nonumber \\
&  \hspace{0.5cm}    
       -  \frac{1}{6} \K ( \varrho   +  3 p ) \de_{\bar{\al}\bar{\be}}
   \Big ] =   0 
   \label{reduced_eq_Ga_0}
   \\     
& \partial_t E_{\bar{\al} \bar{\be}} 
   + E_{\bar{\mu} \bar{\be}} 
   \Ga_{\bar{\al} \msm 0}^{\ms \bar{\mu}}  + E_{\bar{\al} \bar{\mu} }
   \Ga_{\bar{\be} \msm 0}^{\ms \bar{\mu}} + 
   \D_{\bar{\mu}} B_{\bar{\nu} ( \bar{\al}}^{} \ve_{\bar{\be} ) }^{\ms\bar{\mu}\bar{\nu}}
   + 2 \de_{\bar{\mu} \bar{\la} } \Ga_{0 \msm 0}^{\ms \bar{\la}}  \ve^{\bar{\mu}\bar{\nu}}_{\msb (\bar{\al}}
  B_{\bar{\be} ) \bar{\nu}}^{} 
     \nonumber \\
   &  \hspace{0.5cm}
 -3 \de^{\bar{\mu}\bar{\la}} \Ga_{( \bar{\al} \ms |\bar{\la}|}^{\msb 0} E_{\bar{\be} )
   \bar{\mu}}^{}
   - 2\de^{\bar{\mu} \bar{\la}} \Ga_{\bar{\la} \ms (\bar{\al}}^{\ms 0} E^{}_{\bar{\be} )\bar{\mu}} 
 +
   \de_{\bar{\al}\bar{\be}} \de^{\bar{\mu} \bar{\la}} \de^{\bar{\nu} \bar{\si}} 
   \Ga_{\bar{\la} \msm \bar{\si}}^{\ms 0} E_{\bar{\mu} \bar{\nu}}  \nonumber \\ 
& \hspace{0.5cm}  
     +2 \de^{\bar{\mu}\bar{\nu}} \Ga_{\bar{\mu} \msm \bar{\nu}}^{\ms 0} E_{\bar{\al}\bar{\be}} 
   - \frac{1}{4} \K (p + \varrho) (\Ga_{\bar{\al}\msm \bar{\be}}^{\ms 0} +
    \Ga_{\bar{\be}\msm \bar{\al}}^{\ms 0} 
      \nonumber \\
   &  \hspace{0.5cm}
      - \frac{2}{3} \de_{\bar{\al}\bar{\be}}
     \de^{\bar{\mu}\bar{\nu}}
     \Ga_{\bar{\mu} \msm \bar{\nu}}^{\ms 0} ) = 0 
     \label{reduced_eq_E}
     \\ 
&  \partial_t B_{\bar{\al} \bar{\be}} 
  + B_{\bar{\mu} \bar{\be}} 
   \Ga_{\bar{\al} \msm 0}^{\ms \bar{\mu}}  + B_{\bar{\al} \bar{\mu} }\Ga_{\bar{\be} \msm 0}^{\ms \bar{\mu}} -
   \D_{\bar{\mu}} E_{\bar{\nu} ( \bar{\al}}^{} \ve_{\bar{\be} ) }^{\ms\bar{\mu}\bar{\nu}}
   - 2 \de_{\bar{\mu} \bar{\la} } \Ga_{0 \msm 0}^{\ms \bar{\la}}  \ve^{\bar{\mu}\bar{\nu}}_{\msb (\bar{\al}}
  E^{}_{\bar{\be} ) \bar{\nu}} 
  \nonumber \\
&   \hspace{0.5cm} 
   - \de^{\bar{\mu} \bar{\la}} \Ga_{\bar{\la}\msm (\bar{\al}}^{\ms 0} B_{\bar{\be}) \bar{\mu}}^{}
  - 2 \de^{\bar{\mu} \bar{\la} }\Ga_{ (\bar{\al} \msm |\bar{\la}| }^{\msm 0} B_{\bar{\be} ) \bar{\mu}}^{}
  + \de^{\bar{\mu} \bar{\nu}} \Ga_{\bar{\mu} \msm \bar{\nu}}^{\ms 0} B_{\bar{\al}\bar{\be} }
  \nonumber \\
&   \hspace{0.5cm}   
  - \Ga_{\bar{\mu}\msm \bar{\nu}}^{\ms 0} B_{\bar{\si} \bar{\la}} \ve^{\bar{\si}\bar{\mu}}_{\msb (\bar{\al}}
  \ve^{\bar{\nu}\bar{\la}}_{\msb \bar{\be})} = 0 
  \label{reduced_eq_B}
  \\ 
&
  \partial_t \varrho + (p + \varrho ) \Ga_{\mu \msm 0}^{\ms \mu }= 0  
  \label{reduced_eq_rho}
  \\ 
&  \partial_t s = 0  \label{reduced_eq_s} \\ 
&   \partial_t s_\al - (\Ga_{0 \msm \al}^{\ms \mu} - \Ga_{\al \msm 0}^{\ms \mu} )s_\mu = 0 
\label{reduced_eq_s_al}
\\ 
& \partial_t r + r \Ga_{\mu \msm 0}^{\ms \mu} = 0, 
\label{reduced_eq_r}
\end{align}
\end{subequations}
where it is understood that in the expressions involving partial derivatives
of $\varrho$ with respect to the matter variables, 
$\varrho$ is to be replaced by $\cP$ (since 
(\ref{eq_of_state}) is not a part of the above system); 
 $p$ and $\nu^2$ are given by (\ref{pressure}) and (\ref{sound_speed}), respectively;
and $(\cdot)_{[\al|\mu|\be]}$
means  anti-symmetrization of the indices
$\al$ and $\be$ only.
The unknowns to be determined are $e^A_{\ms\bar{\al}}$,
$\Ga_{\bar{\al}\ms \bar{\ga} }^{\ms \bar{\be} }$, $\Ga_{0 \ms \bar{\al}}^{\mss 0}$, 
$\Ga_{\bar{\al}\ms \bar{\ga} }^{\ms 0 }$, 
$E_{\bar{\al}\bar{\be}}$, $B_{\bar{\al}\bar{\be}}$,
$\varrho$, $r$, $s$, $s_\al$. It is explicitly assumed that $E$ and $B$ are symmetric,
 therefore equations (\ref{reduced_eq_E}) and (\ref{reduced_eq_B}) are written
only for the independent components of these fields, say $\bar{\al} \leq \bar{\be}$,
with the remaining components defined by these symmetry relations.
Similarly, all the symmetries (\ref{gamma_2_indices}), (\ref{gamma_sym_two_bar}), 
(\ref{gamma_sym_spatial}) and 
(\ref{gamma_sym_two_0}) and the gauge condition $\Ga_{0\ms \bar{\al}}^{\ms \bar{\be}} = 0$
from lemma \ref{gauged_quantities_lemma} are assumed, with equations (\ref{reduced_eq_Ga_bar}),
(\ref{reduced_eq_Ga_0_0}) and (\ref{reduced_eq_Ga_0}) written only for the independent components
and the remaining ones being defined by these symmetry relations. 
We write the unknowns collectively as a vector
\begin{gather}
z = 
(e^A_{\ms\bar{\al}},
\Ga_{\bar{\al}\ms \bar{\ga} }^{\ms \bar{\be} }, \Ga_{0 \ms \bar{\al}}^{\mss 0}, 
\Ga_{\bar{\al}\ms \bar{\ga} }^{\ms 0 }, 
E_{\bar{\al}\bar{\be}}, B_{\bar{\al}\bar{\be}}, \varrho, r, s, s_\al).
\label{unk_reduced}
\end{gather}

\begin{rema}
In order that some components of $\Ga$, $E$ and $B$ be defined by symmetry, as mentioned above, 
it is necessary that the initial conditions also obey 
such relations. Given arbitrary $\left. E_{\bar{\al}\bar{\be}} \right|_{t=0}$,
$1 \leq \bar{\al} \leq \bar{\be} \leq 3$, one
can in principle always define the remaining $\left. E_{\bar{\al}\bar{\be}} \right|_{t=0}$ by
$E_{\bar{\al}\bar{\be}} = E_{\bar{\be}\bar{\al}}$. However, in the case of interest, $E$ has 
to be determined from an initial data set, in which case all its components 
will be given on the initial slice (see proposition \ref{prop_initial_data} below).
In this situation, the symmetry of $E$ at $t=0$ has to be demonstrated --- and 
only then 
 is one allowed to the write system (\ref{reduced_system}) solely for  
$\{E_{\bar{\al}\bar{\be}}\}_{\bar{\al}\leq\bar{\be}}$
and impose symmetry relations for the remaining components. Similar 
statements hold for the other fields
involving symmetries.
\end{rema}

\section{Proof of theorem \ref{main_theorem}. \label{proof_main_theorem_section}}

\subsection{Initial data. \label{initial_data_section}} In order to address the 
solvability of  the systems (\ref{reduced_system}), we need to provide suitable initial conditions.
These should be determined entirely by the initial data for the Einstein-Euler-Entropy system 
(as this is the set of equations we are ultimately interested in) and our gauge choices. 
We also have to show that the initial data for $E$ and $B$, which are naturally constructed
from the initial data set, are indeed symmetric and trace-free.
Although all of this can be  inferred from 
similar works treating the vacuum and conformal vacuum Einstein equations, as well
as their coupling to the Yang-Mills equations
\cite{FriHypGauge, Friedrich1, Friedrich2, Friedrich3, FriNagy}, an explicit proof does not seem to be available in the 
literature in the case of our system with our gauge choices. It is therefore 
useful to provide it here.

\begin{prop}
Let $\cI = (\Si, g_0, \kappa, r_0, s_0, v, \ccP)$ be an initial data set for the Einstein-Euler-Entropy 
system, with Einsteinian development $(\cM,g)$ that is a perfect fluid source where 
the Einstein-Euler-Entropy system of equations is satisfied. Let 
$\{ e_\al \}_{\al=0}^3$ be a fluid source gauge defined on a coordinate chart $U$ of $\cM$, and 
let $z$ be as in  (\ref{unk_reduced}). Then 
\begin{gather}
\left. z\right|_{\Si \cap U}
\nonumber
\end{gather}
can be written in terms of quantities determined entirely by $\cI$. Furthermore,
$\left. E \right|_{\Si}$ and $\left. B \right|_{\Si}$ are symmetric and trace-free.
\label{prop_initial_data}
\end{prop}
\begin{proof}
It will be useful to first express space-time quantities in terms of an adapted frame.
Let $\{ \we_{\bar{\al}} \}^3_{\bar{\al}=1}$ be 
a frame on $\Si$ orthonormal with respect to $g_0$. Denote by $\we_0$ the
future directed unit normal (with respect to $g$) of $\Si$. The frame  
$\{ \we_\al \}^3_{\al=0}$ is extended to $\cM$ by parallel transport in the direction of 
$\we_0$. Let $\{ \wx^{\bar{A}} \}_{\bar{A}=1}^3$ be coordinates on $\Si$. These
coordinates are extended to $\cM$ by dragging them along $\we_0$, so that
they are constant on the integral curves of $\we_0$. Denoting by $\wx^0$ the parameter
of such curves, we obtain that $\{ \wx^A \}_{A=0}^3$ is a coordinate system on $\cM$. The
frame and coordinate basis are related by
\begin{gather}
 \we_\al = \we^A_{\ms \al} \frac{\partial}{\partial \wx^A}.
\nonumber
\end{gather}
By construction, it holds that 
\begin{gather}
\we^A_{\ms 0} = \de^A_{\ms 0}, \, \we^0_{\ms \bar{\al}} = 0.
\nonumber
\end{gather}
Quantities expressed in terms of $\we_\al$ and $\wx^A$ will be denoted 
with a tilde $\widetilde{~}\,$, and the same index convention as of notation
\ref{notation_barred_indices} is assumed. In particular, in the frame $\we_\al$,
the metric $g_0$ is written as
\begin{gather}
\widetilde{g}_0{}_{\bar{\al}\bar{\be}} = \operatorname{diag}(-1,-1,-1),
\nonumber
\end{gather}
and the connection coefficients are given by
\begin{gather}
\nabla_{\we_\al} \we_\be = \wGa_{\al \ms \be}^{\ms \ga} \we_\ga.
\nonumber
\end{gather}
Notice that the space-time metric is represented by the (constant) matrix
$\operatorname{diag}(1,-1,-1,-1)$ in both frames $e_\al$ and $\we_\al$. This allows
us to drop the $\widetilde{~}$ from the metric $\widetilde{g}_{\al\be}$, 
but we still write $\widetilde{g}_{\al\be}$ when we want 
to stress that some expression is written in the $\we_\al$ basis.
(The distinction
has to be maintained though when $g$ is written in the coordinate basis 
$\{\frac{\partial}{\partial \wx^A} \}_{A=0}^3$ and 
$\{\frac{\partial}{\partial x^A} \}_{A=0}^3$.)

Arguing similarly to 
 (\ref{gamma_2_indices})-(\ref{gamma_sym_two_0}),
\begin{gather}
\begin{split}
\wGa_{0 \ms \be}^{\ms \ga} =  
\wGa_{\bar{\al} \ms 0}^{\ms 0} = \wGa_{0 \ms \bar{\be} }^{\ms \ga}  = 0, \\
\wGa_{\bar{\al} \ms \bar{\be} }^{\ms 0} = \wGa_{\bar{\al} \ms 0}^{\ms \bar{\be}} =
-\wk_{\bar{\al}\bar{\be}}.
\end{split}
\label{conn_coeff_tilde_gauge}
\end{gather}
Now we proceed to relate the tilded quantities to $z$. 
Without loss of generality, it can be assumed that\footnote{Then
on $\Si$ it also holds that $g_{\bar{A}\bar{B}} = \widetilde{g}_{\bar{A}\bar{B}}$, but it will
be convenient to write the $\tilde{~}$ in the metric as a way to keep track of the
changes of basis. } 
\begin{gather}
 \fracbarA =
 \frac{ \partial }{\partial \widetilde{x} } {}_{ \hspace{-0.05cm} \bar{A} }
\text{ on } \Si.
\label{coordinate_equal_Si}
\end{gather}
Write
\begin{gather}
e_{\bar{\al}} = e^0_{\ms\bar{\al}} \frac{\partial}{\partial x^0} + e^{\bar{A}}_{\ms \bar{\al} }
\fracbarA .
\nonumber
\end{gather}
On $\Si$, $e^{\bar{A}}_{\ms \bar{\al} } \fracbarAline$ can be written in terms of $\we_{\bar{\al}}$, 
which in turn is determined by $g_0$. As for $e^0_{\ms \bar{\al}}$, compute 
\begin{align} 
\begin{split}
& \langle e_{\bar{\al}}, e_{\bar{\al}} \rangle
= -
1 =\langle 
 e^{A}_{\ms \bar{\al} }\frac{\partial}{\partial x^A}, 
 e^{B}_{\ms \bar{\al} }\frac{\partial}{\partial x^B} \rangle 
 \\
& =
 (e^0_{\ms\bar{\al}})^2 
  + 2 e^{\bar{A}}_{\ms \bar{\al}} e^0_{\ms \bar{\al}} g_{\bar{A}0} +
 e^{\bar{A}}_{\ms \bar{\al}} e^{\bar{B}}_{\ms \bar{\al}} g_{\bar{A}\bar{B}},
 \end{split}
\label{e_bar_al_squared}
\end{align}
where we have used that $g_{00} = \langle e_0, e_0 \rangle \equiv \langle 
\frac{\partial}{\partial x^0}, \frac{\partial}{\partial x^0} \rangle = 1$.
But since the projection of $u$ onto $\Si$ is $v$,  it holds that 
\begin{gather}
v_{\bar{\al}} = \langle u, \we_{\bar{\al}}  \rangle
= \langle u^B \frac{\partial}{\partial x^B } ,  \we^{\bar{A}}_{\ms \bar{\al}} \fracbarA \rangle
= \we^{\bar{A}}_{\ms \bar{\al}} g_{0\bar{A}} \, \text{ on } \Si,
\label{v_g}
\end{gather}
where  $u^A = \de^A_{\ms 0}$ and (\ref{coordinate_equal_Si}) have been used. 
From (\ref{v_g}), we can now get
$g_{0\bar{A}}$ on $\Si$. Therefore, restricting (\ref{e_bar_al_squared}) 
to $\Si$ allows us to solve for 
for $\left. e^0_{\ms\bar{\al}} \right|_{\Si}$ since $e_{\bar{\al}}$ is space-like
 and unit, with the sign of $e^0_{\ms\bar{\al}}$ being unambiguously defined by our choice
 of orientation. Therefore, $\left. e^0_{\ms \bar{\al}} \right|_\Si$ is also determined by 
the initial data.
Because $e^A_{\ms 0} = \de^A_{\ms 0}$, we conclude that all the functions $e^A_{\ms \al}$
are determined on $\Si$ by the initial data (and of course our gauge choice), the same being
true for $g_{AB}$ in light of (\ref{v_g}) and $g_{00} = 1$.

The change of basis from $\{ \we_\al \}$ to $\{ e_\al \}$ is given by a Lorentz 
transformation $\La$:
\begin{gather}
e_\al = \La^\mu_{\ms \al} \we_\mu.
\label{Lorentz}
\end{gather}
Taking the
inner product of  (\ref{Lorentz}) with $\fracbarAline$ and restricting to $\Si$ produces
\begin{gather}
e^B_{\ms \al} g_{B \bar{A} } = \La^{\bar{\mu}}_{\ms \al} \we^{\bar{B}}_{\ms \bar{\mu}} 
\widetilde{g}_{\bar{B} \bar{A} } \, \text{ on } \Si,
\label{Lorentz_Si}
\end{gather}
where $\langle \we_0, \frac{ \partial }{\partial \wx } {}_{ {}^{ \hspace{-0.05cm} \bar{A} } } \rangle = 0$ has 
been used. From our previous relations, it follows that all quantities 
on (\ref{Lorentz_Si}), except possibly the $\Lambda$'s themselves, are determined
by the initial data on $\Si$. Viewing (\ref{Lorentz_Si}) as a system for the 
$\La^{\bar{\mu}}_{\ms \al}$, ($\al=0,\dots,3$, $\bar{\mu}=1,2,3$) (which will
be invertible since the matrix of the system is that of the change from 
$\frac{ \partial }{\partial \wx } {}_{ {}^{ \hspace{-0.05cm} \bar{A} } } $ to 
$\we_{\bar{\mu}}$ up to a lowering of the indices) shows that 
$\left. \La^{\bar{\mu}}_{\ms \al} \right|_{\Si} $ is entirely determined by
the initial data and the gauge choices. Recalling  the identity
\begin{gather}
\La^\al_{\ms \ga} \La^\be_{\ms \de} = g_{\ga \de}
\label{identity_Lorentz}
\end{gather}
then gives
\begin{gather}
 (\La^0_{\ms 0})^2 = 1 + \sum_{\bar{\mu}=1}^3 (\La^{\bar{\mu}}_{\ms 0})^2,
 \nonumber
 \end{gather}
so that $\La^0_{\ms 0}$ is also determined only by the initial data when restricted to $\Si$.
Notice that the sign of $\La^0_{\ms 0}$ is positive because $\La$ belongs to
the proper Lorentz group\footnote{Since $u$ belongs to  the inside of the future light-cone.}.
From (\ref{identity_Lorentz}), the remaining $\La$'s are given in terms of 
$\La^{\bar{\mu}}_{\ms \al}$ and we conclude that $\left. \La \right|_{\Si}$ is determined
by the initial data.

Next we investigate the connection coefficients. 
Taking the inner products of (\ref{conn_coeff}) with $e_\de$,
\begin{gather}
 g_{\de\ga} \Ga_{\al \ms \be}^{\ms \ga} = \langle \nabla_\al e_\be, e_\de \rangle. 
\end{gather}
Writing the frames on the right hand
side of the above expression in terms of $\we_\al$ via (\ref{Lorentz}) leads to
\begin{gather}
\Ga_{\al \ms \be}^{\ms \xi} 
=g_{\nu\si} \La^\si_{\ms \ga} g^{\ga\xi} \La^\mu_{\ms \al} \we_\mu(\La^\nu_{\ms \be})
+ g_{\tau\si}\La^\si_{\ms \ga} g^{\ga \xi} \La^\nu_{\ms \be} \La^\mu_{\ms \al} 
\wGa_{\mu \ms \nu}^{\ms \tau}.
\label{transformation_connection}
\end{gather}
On $\Si$, the coefficients $\wGa_{\bar{\mu} \ms \bar{\nu}}^{\ms \bar{\tau}}$ 
are determined by $g_0$. Then, by virtue of 
(\ref{conn_coeff_tilde_gauge}), the construction of $\we_0$ and $\we_{\bar{\mu}}$,
 and our previous relations involving 
$\La$, it follows that all terms on the right hand side of 
(\ref{transformation_connection}) are determined by the initial data, except
possibly those involving the derivatives of the Lorentz transformation
in the direction of $\we_0$. To see how such terms are determined, recall that 
in fluid source gauge $\Ga_{0 \ms \bar{\be}}^{\ms \bar{\xi}} = 0$, so that (\ref{transformation_connection}) gives
\begin{gather}
0
= \La^0_{\ms \ga} g^{\ga\bar{\xi}} \La^0_{\ms 0} \we_0(\La^0_{\ms \bar{\be}})
 + g_{\bar{\nu} \si}\La^\si_{\ms \ga} g^{\ga\bar{\xi}} \La^{\bar{\mu}}_{\ms 0} \we_{\bar{\mu}}(\La^{\bar{\nu}}_{\ms \bar{\be}})
 \nonumber \\
 + g_{\tau\si}\La^\si_{\ms \ga} g^{\ga \bar{\xi}} \La^\nu_{\ms \bar{\be}} \La^\mu_{\ms 0} 
\wGa_{\mu \ms \nu}^{\ms \tau}.
\nonumber
\end{gather}
Because $\La^0_{\ms 0} \geq 1$, we see that 
\begin{gather}
\left. \La^0_{\ms \ga} g^{\ga\bar{\xi}} \we_0(\La^0_{\ms \bar{\be}}) \right|_{\Si}
\label{der_0_La_1} 
\end{gather}
can be written in terms of quantities determined by $\cI$. Using then (\ref{der_0_La_1}) 
into (\ref{transformation_connection}) with $\al \mapsto \bar{\al}$,
$\be\mapsto \bar{\be}$ and $\xi \mapsto \bar{\xi}$ shows that 
$\left. \Ga_{\bar{\al} \ms \bar{\be}}^{\ms \bar{\xi}} \right|_{\Si}$ is also written solely in terms of quantities
coming from $\cI$. The coefficients $\left. \Ga_{0 \ms \bar{\be}}^{\mss 0}\right|_{\Si}$
and $\left. \Ga_{\bar{\al}  \ms \bar{\be}}^{\ms 0}\right|_{\Si}$
are similarly determined: from $\Ga_{0\ms 0}^{\mss 0} =0$ and (\ref{transformation_connection}) one
obtains that $\left. \we_0(\La^0_{\ms 0}) \right|_{\Si}$ is determined by $\cI$; using this 
fact into the expressions for $\Ga_{0 \ms 0}^{\mss \bar{\xi}}$ and 
$\Ga_{\bar{\al} \ms 0}^{\ms \bar{\xi}}$ shows that the same is true for 
these quantities restricted to $\Si$, and hence the claim follows upon evoking
(\ref{gamma_sym_two_bar})
and 
(\ref{gamma_sym_two_0}).

The restrictions of $\varrho$, $s$ and $r$ to $\Si$ have the desired properties in that
they are just scalar functions. For $s_\al$, 
our gauge choice and (\ref{loc_adiabatic}) imply $s_0 \equiv 0$. Then,
recalling that $\frac{\partial}{\partial x^0} = e_0$,
\begin{gather}
 s_{\bar{\al}} = \nabla_{e_{\bar{\al}}} s = 
e^{\bar{A}}_{\ms \bar{\al}}  \frac{ \partial s}{\partial x } {}_{ \hspace{-0.05cm} \bar{A} }, 
\nonumber
\end{gather}
which is determined by $\cI$ on $\Si$ by the above results for $e^{\bar{A}}_{\ms \bar{\al}}$
and noticing that $\left. \frac{ \partial s }{\partial x } {}_{ {}^{ \hspace{-0.05cm} \bar{A} } }\right|_{\Si} = 
\frac{ \partial \scr_0 }{\partial x } {}_{ {}^{ \hspace{-0.05cm} \bar{A} } }$.

From definition \ref{u_elec_u_mag_def},
\begin{gather}
 \widetilde{E}_{\al\be} = \widetilde{W}_{0\al 0\be}\widetilde{u}^0 \widetilde{u}^0 
+ \widetilde{W}_{0\al\bar{\si}\be} \widetilde{u}^0 \widetilde{u}^{\bar{\si}}
+ \widetilde{W}_{\bar{\mu}\al 0\be} \widetilde{u}^{\bar{\mu}}u^0
+ \widetilde{W}_{\bar{\mu}\al\bar{\si}\be}\widetilde{u}^{\bar{\mu}}\widetilde{u}^{\bar{\si}}.
\label{E_al_be_tilde}
\end{gather}
On $\Si$ we have $\widetilde{u}^{\bar{\mu}} = \widetilde{v}^{\bar{\mu}}$, while
$\left. \widetilde{u}^0 \right|_{\Si}$ is computed from the $\widetilde{v}$ and the
normalization condition $u^\al u_\al = 1$. Therefore, to show that 
$\left. \widetilde{E}_{\al\be} \right|_{\Si}$ is written in terms of quantities
determined by $\cI$, we only need to investigate the components of the Weyl tensor in 
(\ref{E_al_be_tilde}), and by its symmetries, it suffices to do so for components 
of the form $\widetilde{W}_{0 \bar{\be} 0 \bar{\de}}$, $\widetilde{W}_{0 \bar{\al} \bar{\si} \bar{\de}}$ and
$\widetilde{W}_{\bar{\mu} \bar{\al} \bar{\si} \bar{\be}}$.

From the Gauss equation, the decomposition of the Riemann tensor (i.e., (\ref{decomposition}) 
with $d_{\al\be\ga\de} \equiv 0$), Einstein equations, and (\ref{trace_Einstein_eq}), we 
obtain
\begin{align}
\begin{split}
 \widetilde{W}_{0 \bar{\be} 0 \bar{\de}} = & 
\frac{1}{2} \K \widetilde{T}_{\bar{\be}\bar{\de}} - \frac{1}{6} \K \widetilde{T} \, \widetilde{g}_{\bar{\be}\bar{\de}}
-\frac{1}{2} \K \widetilde{T}_{00} \widetilde{g}_{\bar{\be}\bar{\de}}
\\
&  - {}^{(3)}\widetilde{R}_{\bar{\be}\bar{\de}} - \wk_{\bar{\la}}^{\mss \bar{\la}} \wk_{\bar{\be}\bar{\de}}
+ \wk^{\bar{\la}}_{\ms \bar{\de}} \wk_{\bar{\be}\bar{\la}} \, \text{ on } \Si,
\end{split}
\label{Weyl_constraint_1}
\end{align}
and
\begin{align}
\begin{split}
\widetilde{W}_{\bar{\mu}\bar{\al}\bar{\si}\bar{\be}} & = {}^{(3)} \widetilde{R}_{\bar{\mu}\bar{\al}\bar{\si}\bar{\be}}
+ \wk_{\bar{\mu}\bar{\si}}\wk_{\bar{\al}\bar{\be}} - \wk_{\bar{\mu}\bar{\be}}\wk_{\bar{\al}\bar{\si}}
-\frac{1}{2} \K (\widetilde{T}_{\bar{\be}\bar{\al}} - \frac{1}{3} \widetilde{T} \, \widetilde{g}_{\bar{\be}\bar{\al}} ) 
\widetilde{g}_{\bar{\mu}\bar{\si}} 
\\
&  + \frac{1}{2} \K ( \widetilde{T}_{\bar{\si}\bar{\al}} - \frac{1}{3} \widetilde{T} \,
\widetilde{g}_{\bar{\si}\bar{\al}} ) \widetilde{g}_{\bar{\mu}\bar{\be}} 
 + \frac{1}{2} \K ( \widetilde{T}_{\bar{\be}\bar{\mu}} - \frac{1}{3} \widetilde{T} \, \widetilde{g}_{\bar{\be}\bar{\mu}} ) 
\widetilde{g}_{\bar{\al}\bar{\si}}
\\
&
-\frac{1}{2} \K ( \widetilde{T}_{\bar{\si}\bar{\mu}} - \frac{1}{3} \widetilde{T} \, \widetilde{g}_{\bar{\si}\bar{\mu}} )
\widetilde{g}_{\bar{\al}\bar{\be}} \, \text{ on } \Si,
\end{split}
\label{Weyl_constraint_2}
\end{align}
where ${}^{(3)}\widetilde{R}_{\bar{\be}\bar{\de}}$ and 
$ {}^{(3)} \widetilde{R}_{\bar{\mu}\bar{\al}\bar{\si}\bar{\be}}$ are respectively 
the Ricci and Riemann curvature of $(\Si, g_0)$.
In a similar fashion but using now the Codazzi equation:
\begin{gather}
 \widetilde{W}_{0\bar{\al}\bar{\si}\bar{\be}} =
{}^{(3)} \nabla_{\bar{\si}} \wk_{\bar{\be}\bar{\al}} - 
{}^{(3)} \nabla_{\bar{\be}} \wk_{\bar{\si}\bar{\al}} + \frac{1}{2} \K \widetilde{T}_{0\bar{\be}} 
\widetilde{g}_{\bar{\al}\bar{\si}}
- \frac{1}{2} \K \widetilde{T}_{0\bar{\si}} \widetilde{g}_{\bar{\al}\bar{\be}} \, \text{ on } \Si,
\label{Weyl_constraint_3}
\end{gather}
where ${}^{(3)} \nabla$ is Levi-Civita connection of $g_0$.

From (\ref{perfect_fluid_source}), (\ref{eq_of_state}), (\ref{pressure}), (\ref{E_al_be_tilde}), 
(\ref{Weyl_constraint_1}), (\ref{Weyl_constraint_2}),
and (\ref{Weyl_constraint_3}), we obtain that $\left. \widetilde{E}_{\al\be} \right|_{\Si}$ is 
a tensor solely determined by $\cI$, which is symmetric and trace-free by the
constraint equations (\ref{hamilonian_constraint}) and (\ref{momentum_constraint}).
The same holds for $\left. E_{\al\be} \right|_{\Si}$ by the invariance of the trace and
the properties of $\La^\mu_{\ms \nu}$ previously shown. By an analogous argument,
a similar statement holds for $\left. B_{\al\be} \right|_{\Si}$.
\end{proof}

\begin{defi}
By proposition \ref{prop_initial_data}, given an initial data set, a choice of 
(fluid source) gauge uniquely determines initial conditions for the reduced system.
These initial conditions for the reduced system are henceforth
called a \Em{reduced initial data set}.
\end{defi}

For practical applications, e.g., to numerically solve the equations,
proposition \ref{prop_initial_data} tells us  how to arrange the initial data.
Given a negative three dimensional Riemannian manifold $(\Si,g_0)$, choose coordinates
$\{ \frac{ \partial }{\partial \widetilde{x} } {}_{ {}^{ \hspace{-0.05cm} \bar{A} } } \}_{\bar{A}=1}^3$ 
and an orthonormal frame $\{ \we_{\bar{\al}} \}_{\bar{\al}=1}^3$. 
Declare a metric on $[0,T]\times \Si$ by $g = d\,\widetilde{t}\,{}^2 + g_0$,
where the tilde emphasizes that we identify $[0,T]$ with the time coordinate
$\widetilde{x}^0$ (so that $\we_0 = \frac{\partial}{\partial \widetilde{x}^0}$), and 
the metric is written in terms of the resulting coordinates on 
$[0,T]\times \Si$.
Using this metric and the (given) 
three-velocity $v$ we determine the four velocity $u$ on $\Si$
by the condition $\langle u, u\rangle = 1$ (notice 
that flowing along $u$ produces the parameter $t$, i.e., $x_0$).
Besides, we complete $u$ to an orthonormal frame $\{ e_\al \}_{\al=0}^3$ at the 
instant $\widetilde{x}^0 = 0$.  
The relations given in proposition \ref{prop_initial_data} can now be used to 
\emph{define} the remaining quantities such as $\La^0_{\ms0}$, 
$\wGa_{\al \ms \be}^{\ms \xi}$, $\Ga_{\al \ms \be}^{\ms \xi}$, etc., 
on the initial Cauchy surface.

\subsection{Well-posedness of the reduced system. \label{well_posedness_reduced_section}}
Short-time existence for the reduced system is a direct consequence of the
way it has been set up, at least under our hypotheses. In fact, the gauge choice and construction of
(\ref{reduced_system}), originally devised by Friedrich \cite{Fri}, are motivated exactly by the
attempt of obtaining a reduced system that is symmetric hyperbolic, in which 
case well-known results can be applied. There are, however, one subtlety and one 
observation, that have to be dealt with. First, the initial conditions for the $e^A_{\ms\bar{\al}}$, 
$\Ga$, $E$ and $B$ involve a different number of derivatives of the initial metric $g_0$, 
hence they belong to $H^\ell_{ul}(\Si)$ with different values of $\ell$. The usual
techniques of symmetric hyperbolic systems, however, yield solutions with 
the same regularity for all the unknowns (see e.g. \cite{Majda, TaylorPseudo}), 
which in this case would be that of the less regular initial data, namely, $E$ and $B$.
This does not give the desired differentiability for the frame coefficients, nor for
 the metric. Second, although at first glance the matrix coefficient of 
$\frac{\partial}{\partial t}$ appears to be a diagonal matrix with entries either $1$ or $\nu^2$, 
the 
 ``spatial" derivatives $e_{\bar{\mu}}$ involve derivatives in the direction of $x^0$ hence
 contributing to the  zeroth matrix coefficient. 

\begin{prop}
Let $\cI$ be an initial data set for the Einstein-Euler-Entropy system satisfying the 
hypotheses of theorem \ref{main_theorem}. Fix a positive real number $T_E$ and 
consider the system (\ref{reduced_system})
defined on $[0,T_E] \times \Si$, where $x^0$ is identified with the time coordinate
on $[0,T_E]$. Let $\cI_0$ be the reduced initial data set determined on 
$\{0\} \times \Si$ by $\cI$. Then there exists a unique solution $z$ 
to (\ref{reduced_system}) on some time interval $[0,T_E^\prime]$, $0 <T_E^\prime \leq T_E$,
and satisfying $\left. z \right|_{t=0} = \cI_0$. Furthermore,
 $e^{\bar{A}}_{\ms \bar{\al}} \in C^0([0,T_E^\prime], H_{ul}^{s+1}(\Si)) \cap C^1([0,T_E^\prime], H_{ul}^{s}(\Si))$, 
$ e^{0}_{\ms \bar{\al}} \in C^0([0,T_E^\prime], H_{ul}^{s}(\Si)) \cap C^1([0,T_E^\prime], H_{ul}^{s-1}(\Si))$, 
$\Ga_{\bar{\al}\ms \bar{\ga} }^{\ms \bar{\be} }, \, \Ga_{0 \ms \bar{\al}}^{\mss 0}, \,
\Ga_{\bar{\al}\ms \bar{\ga} }^{\ms 0 } \in 
 C^0([0,T_E^\prime],$ $ H_{ul}^{s}(\Si))  \cap C^1([0,T_E^\prime], H_{ul}^{s-1}(\Si))$,
$E_{\bar{\al}\bar{\be}},\, B_{\bar{\al}\bar{\be}} \in
C^0([0,T_E^\prime], H_{ul}^{s-1}(\Si)) \cap C^1([0,T_E^\prime],$ $H_{ul}^{s-2}(\Si))$,
 $s_\al \in C^0([0,T_E^\prime], H_{ul}^{s-1}(\Si)) \cap C^1([0,T_E^\prime],H_{ul}^{s-2}(\Si))$, 
$\varrho \,$, $r \, $, $s \in C^0([0,T_E^\prime], H_{ul}^{s}(\Si)) \cap C^1([0,T_E^\prime], H_{ul}^{s-1}(\Si))$.
\label{prop_well_pos_reduded}
\end{prop}
\begin{proof}
First, we claim that (\ref{reduced_system}) is symmetric hyperbolic with respect to $t$
on the initial hypersurface and remains so as long as $\nu^2 > 0$, and the slices
$\Si_t = \{ t= \text{ constant} \}$ are space-like with respect to the quadratic form
$g_t$ induced by the frame coefficients.  
Notice that symmetry here means 
symmetry of the matrix coefficients $M^A$ of the derivatives $\frac{\partial}{\partial x^A}$.
Therefore we have  to first change basis via (\ref{frame_coefficients}). 
Other than the first term in each equation, the derivative $\frac{\partial }{\partial t}$ 
also figures in the terms involving $e_{\bar{\mu}}$ in equations (\ref{reduced_eq_Ga_0_0}), 
(\ref{reduced_eq_Ga_0}), (\ref{reduced_eq_E}) and (\ref{reduced_eq_B}), where in these
last two equations the contribution of  $e_{\bar{\mu}}$ comes from the covariant derivatives.
Expressing all derivatives in (\ref{reduced_system}) in terms of $\frac{\partial}{\partial x^A}$ 
 gives that the term in  $\partial_t$
 can be written symbolically as
\begin{align}
 \left( \begin{array}{ccccccccc}
  1 & 0 & 0 & 0 & 0 & 0 & 0 & 0 & 0\\
  0 & 1 & 0 & 0 & 0 & 0 & 0 & 0 & 0\\
  0 & 0  & 1 & -\nu^2 e^0_{\mss\bar{\be}}  & 0 & 0 & 0 & 0 & 0\\
  0 & 0  & -\nu^2 e^0_{\mss\bar{\be}} & \nu^2  & 0  & 0 & 0 & 0 & 0 \\
 0 & 0 & 0 & 0 & m^t_{\bar{\al}\bar{\be}} & 0 & 0 & 0 & 0\\
  0 & 0 & 0 & 0 & 0 & 1 & 0 & 0 & 0\\
  0 & 0 & 0 & 0 & 0 & 0 & 1 & 0 & 0\\
  0 & 0 & 0 & 0 & 0 & 0 & 0 & 1 & 0\\
  0 & 0 & 0 & 0 & 0 & 0 & 0 & 0 & 1
 \end{array}
\right)
\frac{\partial }{\partial t} 
 \left( \begin{array}{c}
e^A_{\mss \bar{\be} } \\
\Ga_{\bar{\de} \ms \bar{\be}}^{\mss\bar{\al}} \\
\Ga_{0 \ms \bar{\al}}^{\mss 0 } \\
\Ga_{ \bar{\al}\mss \bar{\be} }^{\mss 0} \\
(E_{\bar{\al}\bar{\be}},\, B_{\bar{\al}\bar{\be}} ) \\
\varrho \\
s \\
s_\al \\
r
\end{array}
\right) & \nonumber \\
 & \label{matrix_M_0}
\end{align}
where $m^t_{\bar{\al}\bar{\be}}$ is the matrix part corresponding to 
\begin{gather}
 \begin{cases}
  \partial_t E_{\bar{\al}\bar{\be}} 
+ \frac{1}{2} e^0_{\mss \bar{\mu}} \ve_{\bar{\be}}^{\mss \bar{\mu}\bar{\nu}} \partial_t B_{\bar{\nu}\bar{\al}}
+ \frac{1}{2} e^0_{\mss \bar{\mu}} \ve_{\bar{\al}}^{\mss \bar{\mu}\bar{\nu}} \partial_t B_{\bar{\nu}\bar{\be}}, \\
  \partial_t B_{\bar{\al}\bar{\be}} 
- \frac{1}{2} e^0_{\mss \bar{\mu}} \ve_{\bar{\be}}^{\mss \bar{\mu}\bar{\nu}} \partial_t E_{\bar{\nu}\bar{\al}}
- \frac{1}{2} e^0_{\mss \bar{\mu}} \ve_{\bar{\al}}^{\mss \bar{\mu}\bar{\nu}} \partial_t E_{\bar{\nu}\bar{\be}}. 
\label{matrix_M_0_E_B}
 \end{cases}
\end{gather}
From (\ref{matrix_M_0}) and (\ref{matrix_M_0_E_B}), it is seen that 
$M^0 \equiv M^t$ is symmetric. Symmetry of the remaining $M^{\bar{A}}$, $\bar{A}=1,2,3$,
is similarly verified. 

The quadratic form 
$g_t$ is given by
\begin{gather}
{g_t}_{\bar{A}\bar{B}} = f^0_{\ms \bar{A}} f^0_{\ms \bar{B}} - \sum_{\bar{\al}=1}^3 f^{\bar{\al}}_{\ms \bar{A}} 
f^{\bar{\al}}_{\ms \bar{B}},
\label{quadratic_form}
\end{gather}
where the coefficients $f^{\al}_{\ms \bar{A}}$ are defined via
\begin{gather}
\fracbarA
 = f^{\bar{\al}}_{\ms \bar{A}} e_{\bar{\al}}
+ f^0_{\ms \bar{A}} e_0 \equiv f^{\bar{\al}}_{\ms \bar{A}} e_{\bar{\al}}
+ f^0_{\ms \bar{A}} \frac{\partial}{\partial t}.
\nonumber
\end{gather}
From these constructions, we obtain that the characteristics
of the system are non-zero multiples of
\begin{gather}
\xi_0^{K_1}(\xi_0^2 +\frac{1}{4} \pi^{\bar{\al}\bar{\be}} \xi_{\bar{\al}\bar{\be}})^{K_2}
(\xi_0^2 + \nu^2 \pi^{\bar{\ga}\bar{\de}} \xi_{\bar{\ga}\bar{\de}} )^{K_3}
( g^{\la \tau} \xi_\la \xi_\tau )^{K_4},
 \nonumber
\end{gather}
where $K_1,\dots, K_4$ are positive integers. It  follows that the
system is symmetric hyperbolic as long as $\nu^2 > 0$ and $g_t$ remains negative definite.

Consider now the problem on a local patch 
$[0,T_E]\times U$. From our  hypotheses 
and the constructions of proposition \ref{prop_initial_data}, 
 it follows that the initial data $\cI_0  \equiv z(0,\cdot)$ is such that 
\begin{gather}
e^0_{\ms \bar{\al}}(0,\cdot) \in H^{s}(U), 
 e^{\bar{A}}_{\ms \bar{\al}}(0,\cdot) \in H^{s+1}(U), \label{initial_red_frame} \\
(\varrho, r, s)(0,\cdot)  \in H^{s}(U), \label{initial_red_rho} \\
(\Ga_{\bar{\al}\ms \bar{\ga} }^{\ms \bar{\be} }, \Ga_{0 \ms \bar{\al}}^{\mss 0}, 
\Ga_{\bar{\al}\ms \bar{\ga} }^{\ms 0 })(0,\cdot)  \in H^{s}(U), \label{initial_red_Ga} \\
(E_{\bar{\al}\bar{\be}}, B_{\bar{\al}\bar{\be}},  s_\al )(0,\cdot) \in H^{s-1}(U).
\label{initial_red_E}
\end{gather}
$e^{0}_{\ms \bar{\al}}(0,\cdot)$
is only in $H^s$ because it
depends on $v$ (see (\ref{e_bar_al_squared}) and (\ref{v_g})).
From (\ref{initial_red_frame})-(\ref{initial_red_E}), we conclude that 
$\left. \cI_0 \right|_{\Si} \in H^{s-1}$. 
This is enough to apply the theory of quasi-linear symmetric hyperbolic
systems as in \cite{FischerMarsden, KatoQL} (recall that
$s > \frac{3}{2} + 2$), whose hypotheses are satisfied due to the above
positive definiteness of $M^t$. Shrinking $U$ if necessary, we 
 obtain a 
unique solution 
$z_U$ in $C^0([0,T_E^\prime], H^{s-1}(U)) \cap C^1([0,T_E^\prime], H^{s-2}(U))$
for some $0 < T_E^\prime \leq T_E$. 

In order to obtain
the desired regularity, we shall use a bootstrap argument. 
Consider the system for the frame coefficients formed only by equations
(\ref{reduced_eq_frame}), where the $\Ga$'s now enter as 
 inhomogeneous or lower order terms given by  
$z_U$, thus they are in 
$C^0([0,T_E^\prime], H^{s-1}(U)) \cap C^1([0,T_E^\prime], H^{s-2}(U))$. 
This is just a first order symmetric linear system
for the frame coefficients, but there is a mismatch between the initial data,
which is in $H^s$ by (\ref{initial_red_frame}), and the lower order/inhomogeneous 
terms, which are in $H^{s-1}$. The results of Fischer-Marsden \cite{FischerMarsden}
deal precisely with this situation, and we obtain therefore a unique
$C^0([0,T_E^{\prime\prime}], H^{s}(U)) \cap C^1([0,T_E^{\prime\prime}], H^{s-1}(U))$
solution\footnote{The results of \cite{FischerMarsden} apply to a large class
of quasi-linear equations. Allowing 
for less regular lower order terms was, as the authors acknowledge, one
of the goals of the paper.}. 
By uniqueness, this solution agrees with that of $z_U$ 
for $[0, \min\{ T^{\prime\prime}_E, T_E^\prime \} ]$, and shrinking the intervals
if necessary, we can assume $T^\prime_E = T^{\prime\prime}_E$.

Next, consider the system of equations (\ref{reduced_eq_rho}), (\ref{reduced_eq_s}) and (\ref{reduced_eq_r}). 
As before, the initial data, given by (\ref{initial_red_rho})
is in $H^s$, whereas the coefficients are only in $H^{s-1}$. Notice that although 
the matrix  coefficient of $\frac{\partial}{\partial t}$  is the identity, 
depending on the equation of state, 
this system is  semi-linear due to the presence of the pressure
in (\ref{reduced_eq_rho}) (see (\ref{pressure}) and (\ref{eq_of_state})).
In any case, the results of \cite{FischerMarsden} still apply, and we obtain that $\varrho$, $s$ and $r$ are also  in $H^s$.

A similar argument can be applied to the system (\ref{reduced_system}${}^\prime$)
comprised of equations 
(\ref{reduced_eq_Ga_bar}), (\ref{reduced_eq_Ga_0_0}) and (\ref{reduced_eq_Ga_0}). 
The 
$H^{s-1}$ terms $E_{\bar{\al}\bar{\be}}$, $B_{\bar{\al}\bar{\be}}$, $s_\al$,
and the now $H^{s}$ terms 
$\varrho$, $s$ and $r$ given by $z_U$  enter in the system
(\ref{reduced_system}${}^\prime$) as inhomogeneous or lower order terms, 
whereas the initial data for (\ref{reduced_system}${}^\prime$) given by
$\cI_0$ is  in $H^s$ by (\ref{initial_red_Ga}). 
We  write the system once more in terms of the derivatives $\frac{\partial}{\partial x^A}$,
obtaining a semi-linear system where the coefficient 
matrix $N^t$ of $\frac{\partial}{\partial t}$ involves the frame
coefficients $e^A_{\ms \bar{\al}}$ and 
the sound speed $\nu^2$, which  is given in terms of $s$ and $r$ by 
(\ref{sound_speed}). From the above arguments, we obtain that $N^t$ is positive definite
and is in $H^s$. Evoking the results of \cite{FischerMarsden} one more time
gives that $\Ga_{\bar{\al}\ms \bar{\ga} }^{\ms \bar{\be} }, \Ga_{0 \ms \bar{\al}}^{\mss 0}, 
\Ga_{\bar{\al}\ms \bar{\ga} }^{\ms 0 }$ are in fact in 
$C^0([0,T_E^\prime], H^{s}(U)) \cap C^1([0,T_E^\prime], H^{s-1}(U))$.
Using this improved regularity of the connection coefficients 
again with
(\ref{reduced_eq_frame}) and (\ref{initial_red_frame}), for $\bar{A}=1,2,3$, finally gives 
$e^{\bar{A}}_{\ms \bar{\al}} \in C^0([0,T_E^\prime], H^{s+1}(U)) \cap C^1([0,T_E^\prime], H^{s}(U))$.

We now obtain the result on $[0,T_E^\prime] \times \Si$ by a standard
gluing procedure. Uniqueness guarantees that solutions constructed from
different patches $U$and $U^\prime$ agree on the domain of 
dependence\footnote{Defined in the PDE sense.}
 of 
$U\cap U^\prime$.
The time interval $[0,T_E^\prime]$ can be made uniform due to the uniform
conditions on the initial data, and the local in time, global in space, solution $z$
will belong to the desired $H_{ul}^s$ spaces by the way these spaces are constructed out 
of the Sobolev spaces of maps defined on local patches.
\end{proof}
The geometric meaning behind the definiteness of the matrix $M^t$ 
is easy to grasp. If $u$ were hypersurface orthogonal, then $e_{\bar{\al}}$
would be tangent to $\Si_t$, and $M^t$ would be a diagonal matrix with positive entries.
By continuity, we would expect $M^t$ to remain positive definite 
as long as $u$ is sufficiently inside the light-cone.

We also notice that the 
above bootstrap argument for the regularity of some of the components of $z$
works because of the particular form of system (\ref{reduced_system}), which can 
broken in several sub-systems that are ``mildly coupled'' among themselves. 
Since several of the quantities involved have direct physical meaning, 
it would be interesting to see if such split into sub-systems can have
an useful physical interpretation, perhaps in terms of some effective notion
of weak coupling among certain quantities. 


\begin{coro}
 The solutions $E$ and $B$ constructed in proposition 
\ref{prop_well_pos_reduded} are trace-free, and $\varrho$
satisfies (\ref{eq_of_state}).
\label{coro_trace_eq_state}
\end{coro}
\begin{proof}
 Tracing equations (\ref{reduced_eq_E}) and (\ref{reduced_eq_B}), we obtain
 a first order symmetric hyperbolic system for the traces of $E$ and $B$. Since 
 $\left. E^\al_{\ms \al}\right|_{\Si} = 0 =  \left. B^\al_{\ms \al}\right|_{\Si}$
 by proposition \ref{prop_initial_data} and (\ref{E_0_B_0_gauge}), by uniqueness 
 these tensors remain traceless.
 
Locally $s$ and $r$ are written in terms of the coordinates $s = s(x^0,...\,,x^3)$, 
$r=r(x^0,...\, ,x^3)$.
From (\ref{reduced_eq_rho}), (\ref{reduced_eq_s}), (\ref{reduced_eq_r}),
(\ref{pressure}), we obtain
\begin{gather}
\frac{d}{dt} \ccP( r, s) 
= \partial_t \varrho,
\nonumber
\end{gather}
which implies $\varrho = \ccP(r, s)$ since this holds at $t=0$.
\end{proof}

\subsection{Propagation of the gauge. \label{propagation_gauge}}

Letting $e^A_{\ms 0} = \de^A_{\ms 0}$
and $\cM = [0, T_E^\prime] \times \Si$, from proposition \ref{prop_well_pos_reduded}, 
we obtain a space-time
$(\cM, g)$ with the metric given by (\ref{def_metric_frame_coeff}). Notice 
that $g$ agrees with $g_0$ on $\Si$ because of 
(\ref{quadratic_form}).
$(\cM,g)$ is turned into a fluid source
by setting $u = e_0$. The components of the fields that are not given in proposition
\ref{prop_well_pos_reduded}, e.g. $E_{\al 0}$ etc, are defined by their corresponding
expressions in fluid source gauge, e.g. (\ref{E_0_B_0_gauge}), and their symmetry
relations. The frame $\{ e_\al \}_{\al=0}^3$ is then a fluid source gauge, 
with the quantities of proposition \ref{prop_well_pos_reduded} being exactly field components
written in this gauge. All other quantities \emph{throughout this section will be
written with respect to this frame} unless stated differently. 
Moreover, we shall also assume the hypotheses of theorem \ref{main_theorem}, 
so that the results of the previous section will also be used throughout.
The coordinates
are arranged as explained after definition \ref{def_fluid_source_gauge}; 
in particular $e_0 = \frac{\partial}{\partial t}$.
By construction, (\ref{metric_compatible}) is satisfied
and  the connection associated with $g$ is compatible with the metric, 
but it is not known at this point whether it is torsion free. 
In particular, in all expressions below involving a covariant derivative, 
it is to be understood that $\nabla$ is such a connection and
not the Levi-Civita one, at least until the torsion free condition is demonstrated.
From $E$ and $B$, we define 
$W$ and $W^*$ by  (\ref{decomp_W}) and (\ref{decomp_W_dual}). $W$ has then
the usual symmetries of  Weyl tensor and is trace-free by 
corollary \ref{coro_trace_eq_state}, but it is not
yet known that $W$ is the Weyl tensor of the metric $g$. 
With $u$ and $\varrho$ known and $p$ given by
(\ref{pressure}), we  define
 $T_{\al\be}$ 
 by (\ref{perfect_fluid_source}).
Recall 
(\ref{definition_Riemann})
 and define $d^\al_{\ms \be\ga\de}$
via 
(\ref{decomposition}), where $S_{\al\be}$ in 
(\ref{decomposition})
is given by
  (\ref{Einstein_eq_equivalent}).
Define also
$\cT_{\al\msm \be}^{\mss \ga}$,
$F^\al_{\ms\be\ga\de}$, $F_{\al\be\ga}$ and $q_\al$, by 
(\ref{torsion}), (\ref{Friedrichs_tensor}), (\ref{div_Friedrichs_tensor}) and
(\ref{q_al}), respectively, where $\nu$ in $q_\al$ is given 
by  (\ref{sound_speed}), with $s$, $r$, and $\varrho$ 
 being those of proposition \ref{prop_well_pos_reduded}, which 
 satisfy $\varrho = \ccP(r, s)$ by corollary \ref{coro_trace_eq_state}.
We now proceed to show that (\ref{q_al_equation}), 
(\ref{no_torsion_eq}), (\ref{decomposition_eq}) and (\ref{div_Friedrichs_tensor_eq})
are satisfied. In order to do so, we shall derive a symmetric hyperbolic system 
of equations
for these quantities and show that they vanish on the initial slice.

\begin{rema}
At the risk of being repetitive, we stress again that when referring to 
equations such as (\ref{decomp_W}) and
(\ref{Einstein_eq_equivalent}), it should be understood that they are being used to formally \emph{define}
$W$, $S_{\al\be}$ etc, from the quantities obtained from proposition \ref{prop_initial_data}.
Notice also that because we do not know that $\nabla$ is the Levi-Citiva connection,
 the torsion tensor will have to appear in several manipulations below.
\end{rema}

\begin{lemma}
With the above definitions, $d^\al_{\mss\be\ga\de}$ enjoys all the symmetries
of the Weyl  tensor, 
$\cT_{\al\msm \be}^{\mss \ga} = -\cT_{\be\msm \al}^{\mss \ga}$,
and  (\ref{Euler_1}),
(\ref{loc_adiabatic}) and (\ref{rest_mass_conservation}) are satisfied.
Furthermore, $q_0 = 0$, and 
\begin{gather}
\cT_{0 \msm \bar{\al}}^{\mss \be} = 0, \label{torsion_reduced} \\
 d^{\bar{\al}}_{\ms\bar{\be}0\bar{\ga}} = 0, \label{d_reduced} \\
\nabla^\mu T_{\mu \al} = q_\al .
\label{div_T_q}
\end{gather}
\label{lemma_prep}
\end{lemma}
\begin{proof}
The symmetries of $d^\al_{\mss\be\ga\de}$ and anti-symmetry of the torsion tensor
are direct consequences of their definitions and the fact that $W^\al_{\mss\be\ga\de}$
 has these symmetries.
 In fluid source gauge,
(\ref{torsion_reduced}) and (\ref{d_reduced}) are equivalent to (\ref{reduced_eq_frame})
and (\ref{reduced_eq_Ga_bar}), respectively,  
while (\ref{Euler_1}), (\ref{loc_adiabatic}) and (\ref{rest_mass_conservation}) are
 the same as (\ref{reduced_eq_rho}), (\ref{reduced_eq_s}) and (\ref{reduced_eq_r}), respectively.

Since the pressure is a function of $r$ and $s$ by (\ref{eq_of_state}) and (\ref{pressure}),
differentiating $p$ with respect to $t$ and using (\ref{sound_speed}), (\ref{reduced_eq_s})
and (\ref{reduced_eq_r}) yields
\begin{gather}
\partial_t p + (p+\varrho) \nu^2 \Ga_{\mu \ms 0}^{\ms \mu} = 0.
\label{partial_t_p}
\end{gather}
(\ref{partial_t_p}) is equivalent to $q_0 = 0$ in our gauge.
Computing $\nabla^\mu T_{\mu \al}$ from (\ref{perfect_fluid_source})
and 
 using 
(\ref{reduced_eq_rho})  (or equivalently (\ref{Euler_1})),
\begin{gather}
\nabla^\mu T_{\mu \al} =  
 (p + \varrho) u^\mu\nabla_\mu u_\al + u_\al u^\mu \nabla_\mu p - \nabla_\al p,
\label{div_T_after_reduced}
\end{gather}
which in light of (\ref{partial_t_p}) and our gauge conditions, produces (\ref{div_T_q}).
\end{proof}

The next lemma and the proposition that follows will be the main ingredients in proving the propagation of the gauge.
Although both proofs are heavily computational, they follow the same lines of
\cite{Fri, FriRen, FriNagy}.

\begin{lemma}
The following relations hold:
\begin{subequations}{\label{subsidiary}}
\begin{align}
& \partial_t F_{0 \bar{\al} 0} - \frac{1}{4} \ve_{\bar{\al}}^{\ms \bar{\mu} \bar{\nu} }
\ve_{\bar{\nu}}^{\ms \bar{\la}\bar{\si}} \D_{\bar{\mu}} F_{0 \bar{\la}\bar{\si}} 
+d^{\mu \msb \nu \si}_{\ms [\nu \msm} W_{\si]\mu\bar{\al}0}^{}
- W_{\mu\nu\si[0}^{}d^{\si \ms \mu\nu}_{\mss \bar{\al}]}
\nonumber \\
&\hspace{0.5cm}
- \frac{1}{2} \cT_{\mu \msm \nu}^{\mss \si} \nabla_\si W^{\mu\nu}_{\mss \mss \bar{\al} 0}
 +d^{\mu \msm \nu}_{\ms [0 \mss \bar{\al}]} S_{\mu\nu}
 +d^{\mu \ms \nu}_{\ms \nu \ms [\bar{\al}} S_{0]\mu}^{} +
  \Ga_{\mu\ms 0}^{\mss \mu} F_{0\bar{\al} 0} 
\nonumber \\
&\hspace{0.5cm}
  - \frac{1}{4} \Ga_{\bar{\al} \ms 0}^{\mss \bar{\la}} F_{0\bar{\la} 0} 
  (\nabla_0 \ve^{\bar{\nu}}_{\mss \bar{\al}0}  
  - \frac{1}{2} \nabla^\mu \ve^{\bar{\nu}}_{\mss \mu \bar{\al}} )
    \ve_{\bar{\nu}}^{\mss \bar{\mu}\bar{\si}} F_{0\bar{\mu}\bar{\si}}
    = 0
    \label{sub_1}
\\ 
& \frac{1}{2} \ve_{\bar{\ga}}^{\ms \bar{\mu}\bar{\nu}} \partial_t F_{0\bar{\mu}\bar{\nu}}
  + \frac{1}{2} \ve_{\bar{\ga}}^{\ms \bar{\mu}\bar{\nu}} \D_{\bar{\mu}} F_{0\bar{\nu}0} +
  \frac{1}{2} \ve_{\bar{\ga}}^{\ms \bar{\al}\bar{\be}}(\Ga_{0\ms \bar{\al}}^{\mss 0}F_{0\bar{\be}0}
  - \Ga_{0\ms \bar{\be}}^{\mss 0}F_{0\bar{\al}0} )
\nonumber \\
&\hspace{0.5cm}  
  +\frac{1}{2} \ve_{\bar{\ga}}^{\ms \bar{\al}\bar{\be}} \nabla^\mu \pi_{\mu [\bar{\al}}F_{|0|\bar{\be}]0} 
 -\frac{1}{4} \ve_{\bar{\ga}}^{\ms \bar{\al}\bar{\be}}( \ve^{\bar{\nu}}_{\mss \bar{\al}\bar{\be}} 
  \Ga_{\mu \ms 0}^{\mss \mu}  + \nabla_0 \ve^{\bar{\nu}}_{\mss \bar{\al}\bar{\be}}
    \nonumber \\
& \hspace{0.5cm} 
   +
   \ve^{\bar{\nu}\bar{\mu}}_{\msm [\bar{\al}} \Ga_{|\bar{\mu}|\mss \bar{\be}] }^{\msb 0} )
   \ve_{\bar{\nu}}^{\mss \bar{\si} \bar{\la}} F_{0\bar{\si}\bar{\la}}
   -\frac{1}{4} \K
    \ve_{\bar{\ga}}^{\ms \bar{\al}\bar{\be}}(\nabla^\mu \pi_{\mu \bar{\al}} q_{\bar{\be}} 
    - \nabla^\mu  \pi_{\mu \bar{\be}} q_{\bar{\al}} )
   \nonumber \\
 & \hspace{0.5cm}   
   -\frac{1}{4} \K (p+\varrho) \nu^2 \ve_{\bar{\ga}}^{\ms \bar{\al}\bar{\be}}  d^{0}_{\mss \bar{\al}0\bar{\be}}
= 0     \label{sub_2}
  \\  
  & \partial_t \cT_{\bar{\be} \ms \bar{\ga}}^{\mss \mu}  - \sum_{(0\bar{\be}\bar{\ga})}
   (d^{\mu}_{\ms 0\bar{\be}\bar{\ga}} + \cT_{0 \msm \bar{\be}}^{\mss \la} \cT_{\bar{\ga} \msm \la}^{\mss \mu} ) 
   + \Ga_{0 \ms \la}^{\ms \mu} \cT_{\bar{\be}\msm\bar{\ga}}^{\mss \la} 
   + \Ga_{\bar{\be} \ms \la}^{\ms \mu} \cT_{\bar{\ga}\msm 0}^{\mss \la} 
   + \Ga_{\bar{\ga} \ms \la}^{\ms \mu} \cT_{0 \msm \bar{\be}}^{\mss \la} 
   \nonumber \\
& \hspace{0.5cm}   
    - \sum_{ \langle \bar{\be}\bar{\ga}  \rangle } \Ga_{0 \ms \bar{\be}}^{\mss \la} 
    \cT_{\la \msm \bar{\ga}}^{\mss \mu}
    - \sum_{ \langle \bar{\ga} 0  \rangle } \Ga_{\bar{\be} \ms \bar{\ga}}^{\mss \la} 
    \cT_{\la \msm 0}^{\mss \mu}
    - \sum_{ \langle  0  \bar{\be} \rangle } \Ga_{\bar{\ga} \ms 0}^{\mss \la} 
    \cT_{\la \msm \bar{\be}}^{\mss \mu}
     = 0 \label{sub_3}
  \\ 
 & \partial_t d^{\bar{\mu}}_{\mss \bar{\nu}\bar{\be}\bar{\ga}} 
    +\sum_{(0\bar{\be}\bar{\ga} )} R^{\bar{\mu}}_{\mss \bar{\nu}\si 0 }\cT_{\bar{\be} \ms \bar{\ga}}^{\mss \si} 
    +\frac{1}{2} \ve_{0 \bar{\be} \bar{\ga}}^{\msb\mss \bar{\la}}(F_{\bar{\la}\tau \xi} + g_{\bar{\la} \tau} q_\xi )
    \ve_{\bar{\nu}}^{\mss \bar{\mu} \tau \xi} 
    + \Ga_{0\ms\si}^{\mss \bar{\mu}} d^{\si}_{\mss \bar{\nu}\bar{\be}\bar{\ga}}    
       \nonumber \\
& \hspace{0.5cm}  
    + \Ga_{\bar{\be} \ms\si}^{\ms \bar{\mu}} d^{\si}_{\mss \bar{\nu} \bar{\ga} 0 } 
     + \Ga_{\bar{\ga} \ms\si}^{\mss \bar{\mu}} d^{\si}_{\mss \bar{\nu}  0 \bar{\be}} 
     -\sum_{\langle \bar{\nu}\bar{\be}\bar{\ga} \rangle } \Ga_{0 \ms \bar{\nu}}^{\mss\si} 
     d^{\bar{\mu}}_{\mss \si \bar{\be}\bar{\ga}}
     \nonumber \\
& \hspace{0.5cm}   
      -\sum_{\langle \bar{\nu} \bar{\ga} 0 \rangle } \Ga_{\bar{\be} \mss \bar{\nu}}^{\mss\si} 
     d^{\bar{\mu}}_{\mss \si \bar{\ga} 0 }
      -\sum_{\langle \bar{\nu} 0 \bar{\be} \rangle } \Ga_{\bar{\ga} \mss \bar{\nu}}^{\mss\si} 
     d^{\bar{\mu}}_{\mss \si 0 \bar{\be} } 
     = 0   \label{sub_4}
    \\ 
 & \partial_t d^{\bar{\nu}}_{\ms 0\bar{\be} \bar{\nu} } - \frac{1}{2}
    \ve_{\bar{\be}}^{\ms \bar{\la}\bar{\tau}} \ve_{\tau}^{ \ms \bar{\mu} \bar{\si} }
    e_{\bar{\la}}( d^{0}_{\ms \bar{\mu} 0 \bar{\si} } ) 
    + \sum_{(0\bar{\be}\bar{\ga} )} \de_{\bar{\nu}}^{\mss\bar{\ga}} 
    R^{\bar{\nu}}_{\mss 0 \si 0 }
    \cT_{\bar{\be} \ms \bar{\ga}}^{\mss \si} 
    + \frac{1}{2} \ve_{0 \bar{\be} \bar{\nu}}^{\msb\mss \bar{\la}}(F_{\bar{\la}\tau \xi}
\nonumber \\
& \hspace{0.5cm}
     + 
    g_{\bar{\la} \tau} q_\xi  )  \ve_{0 }^{\mss \bar{\nu}  \tau \xi}
    + \de_{\bar{\nu}}^{\mss \bar{\ga}} \Ga_{0\ms 0}^{\mss \si} d^{\bar{\nu}}_{\mss \bar{\si}\bar{\be}\bar{\ga}}
    + \de_{\bar{\nu}}^{\mss \bar{\ga}} \Ga_{\bar{\be} \mss 0}^{\mss \si} 
    d^{\bar{\nu}}_{\mss \si \bar{\ga} 0 }
    + \de_{\bar{\nu}}^{\mss \bar{\ga}} \Ga_{\bar{\ga} \mss 0}^{\mss \si} 
    d^{\bar{\nu}}_{\mss \si 0 \bar{\be} } 
    \nonumber \\
& \hspace{0.5cm}    
    - \de_{\bar{\nu}}^{\mss \bar{\ga}}  \sum_{\langle \bar{\nu}\bar{\be}\bar{\ga} \rangle} 
    \Ga_{0 \ms \si }^{\mss \bar{\nu}} d^{\si}_{\mss 0  \bar{\be}\bar{\ga}}
     + \de_{\bar{\nu}}^{\mss \bar{\ga}}  \sum_{\langle \bar{\nu}\bar{\ga} 0 \rangle} 
     \Ga_{ \bar{\be} \mss \si }^{\mss \bar{\nu}} d^{\si }_{\mss 0 \bar{\ga} 0 }
    \nonumber \\
& \hspace{0.5cm}       
     + \de_{\bar{\nu}}^{\mss \bar{\ga}}  \sum_{\langle \bar{\nu}0 \bar{\be} \rangle} 
     \Ga_{ \bar{\ga} \mss \si}^{\mss \bar{\nu} } d^{\si }_{\mss 0  0 \bar{\be}}
     = 0   \label{sub_5} 
     \\ 
& \frac{1}{2} \ve_{\bar{\nu}}^{\mss \bar{\al}\bar{\be} }\partial_t d^{0}_{\mss \bar{\al}0\bar{\be}} 
    + \frac{1}{2} \nu^2 \ve_{\bar{\nu}}^{\mss \bar{\al} \bar{\be} } 
    e_{\bar{\al}}( d^{\bar{\la}}_{\mss 0\bar{\be} \bar{\la} } ) 
    +  \frac{1}{2} \nu^2 \pi^{\bar{\al}\bar{\la}} \ve_{(\bar{\la}}^{\msb \bar{\si} \bar{\tau}}
   e_{\bar{\al}}( d^{0}_{\mss \bar{\nu})\bar{\si}\bar{\tau}} )  
   \nonumber \\
 & \hspace{0.5cm} 
       -\frac{1}{6}\nu^2  \ve^{\bar{\be}\bar{\al}\bar{\ga} }
    \Big [  \sum_{(\bar{\al}\bar{\be}\bar{\ga})} R^{0}_{\ms \bar{\nu}\si\bar{\al}} 
   \cT_{\bar{\be} \ms \bar{\ga}}^{\mss \si} 
     + \frac{1}{2} \ve_{\bar{\al}\bar{\be}\bar{\ga}}^{\ms \msm \bar{\la}}
   ( F_{\bar{\la}\tau \xi} + g_{\bar{\la}\tau} q_\xi )  \ve_{\bar{\nu}}^{\mss 0 \tau \xi}  
   \Big ]
   \nonumber \\
 & \hspace{0.5cm}   
    -\frac{1}{6}\nu^2  \ve^{\bar{\be}\bar{\al}\bar{\ga} }
    \Big [ 
    \Ga_{\bar{\al}\mss \si}^{\mss 0} d^{\si}_{\mss \bar{\nu}\bar{\be}\bar{\ga}}
    + \Ga_{\bar{\be}\mss \si}^{\mss 0} d^{\si}_{\mss \bar{\nu}\bar{\ga}\bar{\al}}
    + \Ga_{\bar{\ga}\mss \si}^{\mss 0} d^{\si}_{\mss \bar{\nu}\bar{\al}\bar{\be}}
   \nonumber \\
 & \hspace{0.5cm} 
    - \sum_{\langle \bar{\al}\bar{\be}\bar{\ga} \rangle} \Ga_{\bar{\al}\mss\bar{\nu}}^{\mss \si}
    d^{0}_{\mss \si\bar{\be}\bar{\ga}}
      - \sum_{\langle \bar{\nu}\bar{\ga}\bar{\al} \rangle} \Ga_{\bar{\be}\mss\bar{\nu}}^{\mss \si}
    d^{0}_{\mss \si\bar{\ga}\bar{\al}}
    - \sum_{\langle \bar{\nu}\bar{\al} \bar{\be}\rangle} \Ga_{\bar{\ga}\mss\bar{\nu}}^{\mss \si}
    d^{0}_{\mss \si\bar{\al}\bar{\be}}   
    \Big ]
   \nonumber \\
 & \hspace{0.5cm}     
    +\frac{1}{2} \frac{1}{p + \varrho} \ve_{\bar{\nu}}^{\mss \bar{\al}\bar{\be} }
    \Big[ \frac{1}{2} \Ga_{\si \ms0}^{\mss \la} \cT_{\bar{\al} \ms \bar{\be}}^{\mss \si} q_\la  
    - \frac{1}{2} \pi_{\bar{\al}}^{\mss \bar{\mu}} \pi_{\bar{\be}}^{\mss\bar{\nu}}
    ( d^{\la}_{\ms 0\bar{\mu}\bar{\nu}} + d^{\la}_{\ms \bar{\mu}\bar{\nu} 0} 
       + d^{\la}_{\ms \bar{\nu}0\bar{\mu}} ) q_\la      
     \nonumber \\
 & \hspace{0.5cm} 
       +2 \nu^2  (p+\varrho ) \Ga_{0 \ms \bar{\al}}^{\mss 0 } d^{\bar{\mu}}_{\mss 0 \bar{\be}\bar{\mu}}     
      -2\nu^2  (p+\varrho ) \Ga_{0 \ms \bar{\be}}^{\mss 0 } d^{\bar{\mu}}_{\mss 0 \bar{\al}\bar{\mu}}
        - ( \partial_t p+\partial_t \varrho) d^0_{\mss \bar{\al}0\bar{\be}}
        \nonumber \\
 & \hspace{0.5cm}      
       + \nu^2(p+\varrho) 
         ( \Ga_{\bar{\be} \mss \bar{\al}}^{\mss \bar{\la}} - \Ga_{\bar{\al} \mss \bar{\be}}^{\mss\bar{\la}} )
        d^{\bar{\mu}}_{\mss 0\bar{\la}\bar{\mu}}          
   +(p+\varrho) \Ga_{\bar{\al}\ms 0}^{\mss \bar{\la}} d^{0}_{\mss\bar{\la}0\bar{\be}} 
   -(p+\varrho) \Ga_{\bar{\be}\ms 0}^{\mss \bar{\la}} d^{0}_{\mss\bar{\la}0\bar{\al}} 
    \nonumber \\
& \hspace{0.5cm}    
   + e_{\bar{\al}}\big( \nu^2(p+\varrho) \big ) d^{\bar{\mu}}_{\mss 0 \bar{\be}\bar{\mu}}
   - e_{\bar{\be}}\big( \nu^2(p+\varrho) \big ) d^{\bar{\mu}}_{\mss 0 \bar{\al}\bar{\mu}}
              \Big ] = 0      \label{sub_6}
     \\   
&  \ve_{ (\bar{\al}}^{\msb \bar{\la}\bar{\mu}} 
    \partial_t d^{0}_{\mss \bar{\nu} ) \bar{\la}\bar{\mu} } - e_{(\bar{\al}} 
    (  \ve_{\bar{\nu})}^{ \msb \bar{\la}\bar{\mu} }  d^{0}_{\ms \bar{\la}0\bar{\mu}} ) 
    +\frac{1}{2} \pi_{\bar{\al}\bar{\nu}} \ve^{\bar{\be}\bar{\mu}\bar{\la}} d^{0}_{\mss \bar{\mu}0\bar{\la}}
    e_{\bar{\be}}(\frac{1}{p+\varrho}) 
         \nonumber \\
&  \hspace{0.5cm}    
- \frac{1}{2} \frac{1}{p+\varrho} \pi_{\bar{\al}\bar{\nu}}
     \ve^{\bar{\la}\bar{\mu}\bar{\si} } R_{ \bar{\la}\bar{\mu}\bar{\si}}^{\ms\ms \tau } q_\tau   
    - \frac{1}{2}  \pi_{\bar{\al}\bar{\nu}}  \ve^{\bar{\be}\bar{\mu}\bar{\la}}
    \Big( \nu^2 \Ga_{\bar{\be}\msm \bar{\mu}}^{\ms 0}  
        d^{\bar{\xi}}_{\mss 0 \bar{\la}\bar{\xi}}   -
         \Ga_{\bar{\be}\msm \bar{\mu}}^{\ms \bar{\si} }   d^{0}_{\mss \bar{\si }0\bar{\la}} 
     \nonumber \\
&  \hspace{0.5cm}          
         +
    \nu^2 \Ga_{\bar{\be}\ms \bar{\la}}^{\ms 0 }  
     d^{\bar{\xi}}_{\mss 0 \bar{\mu}\bar{\xi}}   - 
       \Ga_{\bar{\be}\ms \bar{\la}}^{\mss \bar{\si} }  d^{0}_{\mss \bar{\mu}0\bar{\si}} 
     \Big )        
   + \frac{1}{2} \ve_{(\bar{\al}}^{\msb \bar{\be}\bar{\ga}} 
   \sum_{(0\bar{\be}\bar{\ga})} R^0_{\mss\bar{\nu})\si0} \cT^{\mss \si}_{\bar{\be}\msm \bar{\ga}}
      \nonumber \\
& \hspace{0.5cm}  
  + \frac{1}{4} \ve_{(\bar{\al}}^{\msb \bar{\be}\bar{\ga}} \ve_{\bar{\nu})}^{\ms 0 \tau \xi}   \ve_{0 \bar{\be}\bar{\ga}}^{\ms \msm \la}
   ( F_{\la\tau \xi} + g_{\la \tau} q_\xi )      
   + \frac{1}{2} \ve_{(\bar{\al}}^{\msb \bar{\be}\bar{\ga}} \Big ( 
    d^\si_{\mss \bar{\nu}) \bar{\be}\bar{\ga} } \Ga_{0 \ms \si}^{\ms 0}
    \nonumber \\
& \hspace{0.5cm}     
   + d^\si_{\mss \bar{\nu}) \bar{\ga} 0 } \Ga_{\bar{\be}  \ms \si}^{\ms 0}    
   + d^\si_{\mss \bar{\nu})  0 \bar{\be}} \Ga_{\bar{\ga}  \ms \si}^{\ms 0}
   \Big )       
   + \frac{1}{2} \ve_{(\bar{\al}}^{\msb \bar{\be}\bar{\ga}} \Big ( 
    \sum_{ \langle \bar{\nu} \bar{\be}\bar{\ga}\rangle } \Ga_{|0| \ms \bar{\nu} )}^{\msb \si} d^0_{\mss \si \bar{\be}\bar{\ga}}
     \nonumber \\
   & \hspace{0.5cm}    
   + \sum_{ \langle \bar{\nu} \bar{\ga} 0 \rangle } \Ga_{|\bar{\be}| \ms \bar{\nu} )}^{\msb \, \si} d^0_{\mss \si \bar{\ga} 0 }
   + \sum_{ \langle \bar{\nu}  0 \bar{\be} \rangle } \Ga_{|\bar{\ga}| \ms \bar{\nu} )}^{\msb \,\si} d^0_{\mss \si  0 \bar{\be}}
    \Big )
     = 0 \label{sub_7} 
      \\
 & \partial_t q_{\bar{\al}} + \big( \Ga_{\bar{\al}\ms 0}^{\mss \bar{\la} } - 
       \Ga_{0 \ms \bar{\al} }^{\mss \bar{ \la } } \big) q_{\bar{\la}}   
       -2\nu^2(p+\varrho) d^{\bar{\ga}}_{\ms 0\bar{\al} \bar{\ga} } = 0,
         \label{sub_8}  
\end{align}
\end{subequations}
where $(\cdot)_{(\al|\mu|\be)}$ (resp. $(\cdot)_{[\al|\mu|\be]}$)
means symmetrization (resp. anti-symmetrization) of the indices
$\al$ and $\be$ only, $\sum_{(\al\be\ga)}$ indicates
 sum over cyclic permutations of 
$\al\be\ga$,  and
\begin{gather}
\sum_{\langle \al \be \rangle} \Ga_{\tau \ms \al }^{\ms \la } A_{\la \be }
:= \Ga_{\tau \ms \al }^{\ms \la } A_{\la \be } + 
\Ga_{\tau \ms \be  }^{\ms \la } A_{\al \la}, 
\nonumber \\
\sum_{\langle \al \be \ga \rangle} \Ga_{\tau \ms \al }^{\ms \la } A_{\la \be \ga }
:=
\Ga_{\tau \ms \al }^{\ms \la } A_{\la \be \ga } +
\Ga_{\tau \ms \be }^{\ms \la } A_{\al \la  \ga } +
\Ga_{\tau \ms \ga }^{\ms \la } A_{\al \be \la } .
\nonumber
\end{gather}
\end{lemma}
\begin{proof}
We start computing $\nabla^\mu F_{\mu \al \be}$. Commuting the covariant
derivatives and using the symmetries of $W^\mu_{\ms \al\be \ga}$ and $S_{\al \be}$ we find
\begin{align}
\begin{split}
2 \nabla^\mu F_{\mu \al\be}  = &
- R^{\mu \ms \nu \si}_{\mss \nu} W^{}_{\si \mu \al \be}
+ R^{\mu \ms \nu \si}_{\mss \si} W^{}_{\nu \mu \al \be} 
+ 
R^{\si  \ms \mu \nu }_{\mss \al} W^{}_{\mu \nu \si \be}
\\
 & 
- R^{\si  \ms \mu \nu }_{\mss \be} W^{}_{\mu \nu \si \al}  
  + \cT_{\mu \msm \nu}^{\mss \si} \nabla_\si W^{\mu\nu}_{\ms\mss \al\be}
 + \nabla_\al \nabla^\mu S_{\be\mu} - \nabla_\be \nabla^\mu S_{\al \mu}
\\
& 
  - R^{\mu \ms \nu}_{\mss \be \mss \al} S_{\mu\nu} 
 + R^{\mu \ms \nu}_{\mss \al  \mss \be} S_{\mu\nu} 
  - R^{\mu \ms \nu}_{\mss \nu \mss \al} S_{\be \mu} 
 + R^{\mu \ms \nu}_{\mss \nu  \mss \be} S_{\al\mu} 
 \\
 &
  + \nabla_\mu S^{\nu}_{\ms \be} \cT_{\al \msm \nu}^{\mss \mu} 
 - \nabla_\mu S^{\nu}_{\ms \al} \cT_{\be \msm \nu}^{\mss \mu} .
\end{split}
\nonumber
\end{align}
Using (\ref{decomposition}), (\ref{Einstein_eq_equivalent}) and 
(\ref{div_T_q}) and the various symmetries of the tensors involved,
 the above becomes (recall that $T$ is the trace of $T_{\al\be}$)
 \begin{align}
 \begin{split}
 \nabla^\mu F_{\mu \al\be}  = &
  - d^{\mu \msm \nu \si}_{\mss [\nu} W^{}_{\si ]\mu \al \be}
+  W^{}_{\mu \nu \si [\be} d^{\si  \ms \mu \nu }_{\mss \al]}
 + \frac{1}{2}\cT_{\mu \msm \nu}^{\mss \si} \nabla_\si W^{\mu\nu}_{\ms\mss \al\be}
 \\
& 
 -d^{\mu \msm \nu}_{\mss [\be \mss \al]} S_{\mu\nu} 
 - W^{\mu \msm \nu}_{\mss [\be \mss \al]} S_{\mu\nu} 
+ d^{\mu \msm \nu}_{\ms \nu  \mss [\al } S_{\be]\nu}^{} 
  + \frac{1}{2} \K \Big [
\\
&   
   (\nabla_\mu p + \nabla_\mu \varrho) ( \cT_{\al \msm 0}^{\mss \mu} \de_{0\be} 
 + \cT_{\be \msm 0}^{\mss \mu} \de_{0\al} )
 +\nabla_\mu p ( \cT_{\be \msm \al}^{\mss \mu}   - \cT_{\al \msm \be}^{\mss \mu}  ) \Big ]
 \\
 & 
 -\frac{1}{6}\K  ( \cT_{\al \msm \be}^{\mss \mu}   - \cT_{\be \msm \al}^{\mss \mu}  ) \nabla_\mu T
 + \frac{1}{6} \K \cT_{\al \msm \be}^{\mss \mu} \nabla_\mu T + \K \nabla_{[\al} q_{\be]}.
 \end{split}
 \label{div_F_lemma}
 \end{align}
 On the other hand, using (\ref{div_T_q}) 
  into $F_{\mu\al\be}$ gives, after some contractions,
\begin{align}
-\frac{1}{2} \K q_\al = 
\pi_\mu^{\mss \si} \pi_\nu^{\mss \tau} \pi_\al^{\mss \la}  F_{\si\tau\la} \pi^{\mu\nu}
- \pi_\al^{\mss \la} F_{\mu \la \nu} u^\mu u^\nu.
\label{decomp_F_lemma}
\end{align}
 Then, from (\ref{decomp_W}), (\ref{reduced_eq_E}), (\ref{reduced_eq_B}),
 (\ref{reduced_eq_rho}) and 
 (\ref{decomp_F_lemma}), we obtain, after some algebra, 
 \begin{align}
 \begin{split}
 F_{\mu\al\be} = &\,
 \pi_\al^{\mss \la} F_{\si \la \nu} u^\si u^\nu u_\be
 - \pi_\be^{\mss \la} F_{\si \la \nu} u^\si u^\nu u_\al 
 + \frac{1}{2} \pi_{\mu\al} \pi_\be^{\mss \la} F_{\mu \la \nu} u^\mu u^\nu
 \\
 &  
 -  \frac{1}{2} \pi_{\mu\be} \pi_\al^{\mss \la} F_{\mu \la \nu} u^\mu u^\nu
 -\frac{1}{2} u_\mu \ve^\nu_{\mss \al \be} \ve_{\nu}^{\mss \si\tau}
\pi_\si^{\mss \la} \pi_\tau^{\mss \xi} F_{\ga\la\xi}u^\ga
 \\
 &  
+ 
\frac{1}{2}\ve^\nu_{\mss \mu [\al}  u_{\be]}^{} 
\ve_{\nu}^{\mss \si\tau}
\pi_\si^{\mss \la} \pi_\tau^{\mss \xi} F_{\ga\la\xi}u^\ga
- \frac{1}{2} \K \pi_{\mu [\al} q_{\be]}.
\end{split} 
\label{decomp_F_lemma_2}
 \end{align}
Computing $\nabla^\mu F_{\mu\al\be}$ from 
(\ref{decomp_F_lemma_2}), using the  resulting expression into 
(\ref{div_F_lemma}), recalling our gauge conditions,
and evoking lemma \ref{lemma_prep} leads to (\ref{sub_1}) and (\ref{sub_2}) after 
suitably choosing the indices to correspond to the ones of those expressions.

When the torsion of $\nabla$ does not necessarily vanish, the first Bianchi identity
takes the form
\begin{gather}
\sum_{(\al\be\ga)} R^\mu_{\mss \al\be\ga} 
= \sum_{(\al\be\ga)} \Big ( \nabla_\al \cT_{\be \msm \ga}^{\mss \mu}
- \cT_{\al \msm \be}^{\mss\la} \cT_{\ga \msm \la }^{\mss\mu} \Big ).
\label{first_Bianchi_torsion}
\end{gather}
After evoking (\ref{decomp_W}), using  symmetries, 
and setting $\al = 0$, $\be =\bar{\be}$ and $\ga = \bar{\ga}$,
 this expression simplifies
to 
\begin{gather}
\nabla_0 \cT_{\bar{\be} \msm \bar{\ga} }^{\mss \mu} = 
- \nabla_{\bar{\be}} \cT_{\bar{\ga} \msm 0}^{\mss \mu}
- \nabla_{\bar{\ga}} \cT_{0 \msm \bar{\be}}^{\mss \mu}
+
\sum_{(0 \bar{\be} \bar{\ga}) } \Big ( d^\mu_{\mss 0 \bar{\be} \bar{\ga} } 
+ \cT_{0 \msm \bar{\be}}^{\mss\la} \cT_{\bar{\ga} \msm \la }^{\mss\mu} \Big ).
\nonumber
\end{gather}
In view of (\ref{torsion_reduced}), this  gives (\ref{sub_3}).

Taking torsion into account again, the second Bianchi identity reads
\begin{gather}
\sum_{(\al\be\ga)} \nabla_\al R^\mu_{\mss \nu \be\ga} 
+ 
\sum_{(\al\be\ga)} R^\mu_{\mss \nu \la \al } \cT_{\be \msm \ga}^{\mss \la}
=0.
\label{second_Bianchi_torsion}
\end{gather}
From the symmetries of $W^\al_{\mss \be\ga\de}$,
$d^\al_{\mss \be\ga\de}$
and (\ref{first_Bianchi_torsion}), we have the identity
\begin{gather}
\sum_{(\al\be\ga)} \nabla_\al W^\mu_{\mss \nu \be\ga} 
+ \frac{1}{2}\ve^{\la \si \mu}_{\msm\ms\nu} \, \ve_{\al \be\ga}^{\msm\ms\xi }
\, \nabla_\tau W^\tau_{\mss \xi \la \si} = 0.
\label{identity_sym_nabla_W}
\end{gather}
From (\ref{second_Bianchi_torsion}), (\ref{decomposition}), 
(\ref{decomp_W}), (\ref{decomp_W_dual}), 
(\ref{identity_sym_nabla_W}), 
(\ref{Einstein_eq_equivalent}),  and (\ref{div_T_q}), we get, after some simplifications,
\begin{align}
\begin{split}
& \sum_{(\al\be\ga)} \nabla_\al d^\mu_{\mss \nu \be\ga} 
+ \sum_{(\al\be\ga)} R^\mu_{\mss \nu \la \al } \cT_{\be \msm \ga}^{\mss \la}
\\
& \hspace{0.5cm}
 +\frac{1}{2} \ve_{\al\be\ga}^{\ms\ms \la}
 (F_{\la \tau \xi} + g_{\la\tau} q_\xi )\ve_\nu^{\mss \mu \tau \xi} = 0.
 \end{split}
 \label{basic_identity_d_sub}
\end{align}
Setting $\al=0$, $\mu = \bar{\mu}$, $\nu=\bar{\nu}$, $\be = \bar{\be}$ and
$\ga=\bar{\ga}$ in (\ref{basic_identity_d_sub}) and using 
(\ref{d_reduced}) produces (\ref{sub_4}).

From the definition of $q_\al$, compute 
$\nabla_{[\bar{\al}} q_{\bar{\be]}}$, use (\ref{reduced_eq_Ga_0}) 
and $\nu^2 > 0$ to find
\begin{gather}
(p+\varrho)  d^{0}_{\mss \bar{\al} 0 \bar{\be}} 
= -\nabla_{[\bar{\al}} q_{\bar{\be]}},
\label{using_red_1}
\end{gather}
which implies 
\begin{gather}
d^{0}_{\mss \bar{\al} 0 \bar{\be}}  + d^{0}_{\mss \bar{\be} 0 \bar{\al}} = 0.
\label{anti_sym_d_00}
\end{gather}
Then  (\ref{anti_sym_d_00}), (\ref{identity_ve_1}) and the symmetries of 
$d^\al_{\mss\be\ga\de}$
 imply the identity
\begin{gather}
g^{\bar{\ga}\bar{\nu} } e_{\bar{\ga}} (d^0_{\ms \bar{\nu}  0 \bar{\be} } )- 
g^{\bar{\ga}\bar{\nu} }  e_{\bar{\be}} (d^0_{\ms \bar{\nu} 0 \bar{\ga} } ) 
= - \frac{1}{2} \ve_{\bar{\be}}^{\mss\bar{\la}\bar{\tau}} 
e_{\bar{\la}}( \ve_{\bar{\tau}}^{\mss \bar{\mu}\bar{\xi}} d^0_{\mss\bar{\mu}0\bar{\xi}} ).
\label{identity_2_e_e}
 \end{gather}
From  (\ref{basic_identity_d_sub}) with $\al=0$, $\mu = 0$ $\nu=\bar{\nu}$, $\be = \bar{\be}$ and
$\ga=\bar{\ga}$, 
\begin{align}
& 
e_0 (d^0_{\ms \bar{\nu}\bar{\be}\bar{\ga} } ) 
  + e_{\bar{\be}} (d^0_{\ms \bar{\nu} \bar{\ga} 0} )
  + e_{\bar{\ga}} (d^0_{\ms \bar{\nu}  0 \bar{\be} } )
  + \sum_{(0\bar{\be}\bar{\ga} )} R^{0}_{\mss \bar{\nu} \si  0 }
    \cT_{\bar{\be} \ms \bar{\ga}}^{\mss \si}  
    + \Ga_{0\ms \si}^{\mss 0}  d^{\si}_{\mss \bar{\nu}\bar{\be}\bar{\ga}}
 \nonumber \\  
 & 
 \hspace{0.5cm}
     + \Ga_{\bar{\be} \mss \si}^{\mss 0}  d^{\si}_{\mss \bar{\nu}\bar{\ga} 0 }
    + \Ga_{\bar{\ga} \mss \si}^{\mss 0}     d^{\si}_{\mss \bar{\nu}0 \bar{\be} } 
      -   \sum_{\langle \bar{\nu}\bar{\be}\bar{\ga} \rangle} 
    \Ga_{0 \ms \bar{\nu}}^{\mss \si} d^{0}_{\mss \si \bar{\be}\bar{\ga}}
     - \sum_{\langle \bar{\nu}\bar{\ga} 0 \rangle} 
     \Ga_{ \bar{\be} \mss \bar{\nu}}^{\mss \si} d^{0}_{\mss \si \bar{\ga} 0 }
\nonumber    \\
&  
  \hspace{0.5cm}
      -  \sum_{\langle \bar{\nu}0 \bar{\be} \rangle} 
     \Ga_{ \bar{\ga} \mss \bar{\nu}}^{\mss \si} d^{0}_{\mss \si  0 \bar{\be}}
            + \frac{1}{2} \ve_{0 \bar{\be} \bar{\ga}}^{\msb\mss \bar{\la}}(F_{\bar{\la}\tau \xi} + 
    g_{\bar{\la} \tau} q_\xi  )  \ve_{ \bar{\nu}  }^{\mss 0  \tau \xi}  = 0.
    \label{basic_sym_d_expanded}
\end{align}
Using 
 (\ref{gamma_2_indices}), (\ref{gamma_sym_two_bar}), 
(\ref{gamma_sym_spatial}) and 
(\ref{gamma_sym_two_0}), 
observing that $R^0_{\mss\bar{\nu}\si \la } = R^{\bar{\nu}}_{\mss 0 \si \la }$,
$d^0_{\mss\bar{\nu}\si \la } = d^{\bar{\nu}}_{\mss 0 \si \la }$, using 
(\ref{identity_2_e_e}) into  (\ref{basic_sym_d_expanded}), and contracting 
in $\bar{\nu}$ and $\bar{\ga}$ produces
(\ref{sub_5}).

From (\ref{reduced_eq_Ga_0_0})  and the definition of $q_\al$ we find
\begin{gather}
(p+\varrho) \nu^2 d^{\bar{\nu}}_{\mss 0 \bar{\al} \bar{\nu}} = \nabla_{[0} q_{ \bar{\al}] }.
\label{using_red_2}
\end{gather}
From (\ref{using_red_1}) and (\ref{using_red_2}) we then obtain
\begin{gather}
e_0\big( \nabla_{[\bar{\al}} q_{\bar{\be}]} \big)
= -d^0_{\mss \bar{\al}0\bar{\be}} e_0(p+\varrho) 
- (p+\varrho) e_0 (d^0_{\mss \bar{\al}0\bar{\be}} ),
\label{Lie_1}
\\
\nabla_{\bar{\al}} \nabla_{[0} q_{ \bar{\be}] }
= e_{\bar{\al}}\big( \nabla_{[0} q_{ \bar{\be}] } \big )
- \Ga_{\bar{\al} \ms \bar{\be}}^{\ms \la} \nabla_{[0} q_{ \la ] }
- \Ga_{\bar{\al} \ms 0 }^{\ms \la} \nabla_{ [ \la} q_{ \bar{\be}  ]}, 
\label{Lie_2}
\\
\nabla_{\bar{\be}} \nabla_{[0} q_{ \bar{\al}] }
= e_{\bar{\be}}\big( \nabla_{[0} q_{ \bar{\al}] } \big )
- \Ga_{\bar{\be} \ms \bar{\al}}^{\ms \la} \nabla_{[0} q_{ \la ] }
- \Ga_{\bar{\be} \ms 0 }^{\ms \la} \nabla_{ [ \la} q_{ \bar{\al }  ]}.
\label{Lie_3}
\end{gather}
Recall the following identity for the Lie derivative $\cL$
\begin{gather}
\cL_{e_0} \nabla_{[\bar{\al}} q_{\bar{\be}]} 
=  \nabla_0 \nabla_{[\bar{\al}} q_{\bar{\be}]} +
\nabla_{[\bar{\al}} q_{\mu ]} \Ga_{\bar{\be}\ms 0}^{\ms \mu} 
+
\nabla_{[\mu }q_{ \bar{\be} ]} \Ga_{\bar{\al}\ms 0}^{\ms \mu} .
\label{Lie_derivative_q_al}
\end{gather}
Now compute the left hand side of (\ref{Lie_derivative_q_al}) directly from 
(\ref{q_al}), use (\ref{Lie_1}), (\ref{Lie_2}), (\ref{Lie_3}), contract 
the resulting expression with $\ve_{\bar{\nu}}^{\mss \bar{\al}\bar{\be}}$,
and evoke (\ref{basic_identity_d_sub}) once again to find (\ref{sub_6}).

Next, set $\al =0$, $\mu=0$, $\nu=\bar{\nu}$, $\be=\bar{\be}$ and $\ga=\bar{\ga}$ in
(\ref{basic_identity_d_sub}), contract with $\ve_{\bar{\al}}^{\mss\bar{\be}\bar{\ga}}$,
symmetryze on $\bar{\al}$ and $\bar{\nu}$, and use (\ref{using_red_1}) and (\ref{using_red_2})
to obtain (\ref{sub_7}).

Finally, (\ref{sub_8}) follows from (\ref{using_red_2})
and lemma \ref{lemma_prep}.
\end{proof}

Next, we show that the tensors $d^\al_{\mss \be\ga\de}$, $F_{\al\be\ga}$ and 
$\cT_{\al \msm \be}^{\mss \ga}$ vanish on $\cM$.

\begin{prop}
With the above definitions, 
\begin{gather}
 d^\al_{\mss \be\ga\de} = 0, \, F_{\al\be\ga} = 0, \, 
\cT_{\al \msm \be}^{\mss \ga} = 0, \, \text{ and } \, q_{\al} = 0
\nonumber
\end{gather}
on $\cM$.
\label{prop_constraints_vanish}
\end{prop}
\begin{proof}
 In light of our gauge choice, lemma \ref{lemma_prep}, and the symmetries 
involved, several components of the above tensors vanish identically 
on $\cM$. Taking into account the symmetries of the remaining components, it 
is seen that to show the proposition it suffices to prove that 
the components 
\begin{gather}
 \cT_{\bar{\al} \msm \bar{\be}}^{\mss \ga},
\, F_{0 \bar{\al} 0}, \, F_{0 \bar{\al} \bar{\be} }, \,
F_{\bar{\al}\bar{\be}\bar{\ga}}, \, 
 d^0_{\mss  \bar{\al} 0 \bar{\be}}, \,  
d^0_{\mss \bar{\al} \bar{\be} \bar{\ga}}, \,
d^{\bar{\mu}}_{\mss \bar{\al} \bar{\be} \bar{\ga}}, \,
 \, \text{and} \, \, q_{\bar{\al}}
 \label{vanishing_constraints}
\end{gather}
vanish on $\cM$.

Recalling that $F_{\al \be \ga} = - F_{\al \ga \be}$, equations 
(\ref{subsidiary}) can be viewed as a system for the quantities
\begin{gather}
 \cT_{\bar{\al} \msm \bar{\be}}^{\mss \ga},
\, F_{0 \bar{\al} 0}, \, F_{0 \bar{\al} \bar{\be} }, \,
d^{\bar{\mu}}_{\mss \bar{\al} \bar{\be} \bar{\ga}}, \,
d^{\bar{\nu}}_{\mss 0 \bar{\al} \bar{\nu}}, \,
\ve_{\bar{\nu}}^{\mss \bar{\al}\bar{\be} } d^0_{\mss  \bar{\al} 0 \bar{\be}}, \,  
\ve_{(\bar{\nu}}^{\msm \bar{\be}\bar{\ga}} d^0_{\mss \bar{\al}) \bar{\be} \bar{\ga}},
\,\,  \text{and} \,\,
q_{\bar{\al}}.
\label{quantities_system_sub}
\end{gather}
Arguing as in the proof of proposition \ref{prop_well_pos_reduded}, we obtain 
that (\ref{subsidiary}) is a first order symmetric hyperbolic system.
Uniqueness of solutions then implies that 
all quantities in (\ref{quantities_system_sub}) vanish on $\cM$ if they vanish 
on $\Si$. To show that this is the case, it is useful to introduce the adapted
coordinates $\{ \wx^A \}_{A=0}^3$ and frame $\{ \we_\mu \}_{\mu = 0}^3$
as in the proof of proposition \ref{prop_initial_data}, and we shall employ
the same notation and conventions as used 
there\footnote{As some identities of the proof of proposition \ref{prop_initial_data} will be 
evoked as well, it should be noticed that these are still valid
assuming only the hypotheses of this section.}.

From (\ref{torsion_reduced}), it follows that
\begin{gather}
 \widetilde{\cT}_{\al \msm \be}^{\mss \ga} \widetilde{u}^\al = 0.
\nonumber
\end{gather}
Using (\ref{Lorentz}), this gives 
\begin{gather}
 \widetilde{\cT}_{0 \msm \bar{\be}}^{\mss \ga} \La^0_{\mss 0}
+ \widetilde{\cT}_{\bar{\al} \msm \bar{\be}}^{\mss \ga} \La^{\bar{\al}}_{\ms 0} = 0.
\nonumber
\end{gather}
But since the connection on three manifold $(\Si, g_0)$ is torsion free,
we have $\left. \widetilde{\cT}_{\bar{\al} \msm \bar{\be}}^{\mss \bar{\ga}} \right|_\Si = 0$,
whereas $\left. \widetilde{\cT}_{\bar{\al} \msm \bar{\be}}^{\mss 0} \right|_\Si = 0$
by (\ref{conn_coeff_tilde_gauge}). Therefore,
$\left. \widetilde{\cT}_{0 \msm \bar{\be}}^{\mss \ga} \right|_\Si = 0$ since 
$ \La^0_{\mss 0} \neq 0$. We conclude that $\widetilde{\cT}_{\al \msm \be}^{\mss \ga}$, 
and thus $\cT_{\al \msm \be}^{\mss \ga}$, vanishes on $\Si$.

From  (\ref{decomp_F_lemma_2}) and the way
the initial data was constructed in proposition (\ref{prop_initial_data}),
we find, after a somewhat lengthy but not difficult calculation, that
\begin{gather}
F_{0\bar{\al}0} = 0 = F_{0 \bar{\al}\bar{\be}} \text{ on } \Si.
\label{use_constraints_vanish_F}
\end{gather}
Using corollary \ref{coro_trace_eq_state} and lemma \ref{decomp_W_lemma}, we obtain,
with the help of (\ref{div_T_q}),
\begin{gather}
g^{\al\be} F_{\al\be\ga} = -\frac{1}{2}\K q_\ga.
\nonumber
\end{gather}
In light of (\ref{use_constraints_vanish_F}) and (\ref{decomp_F_lemma_2}),
this gives 
\begin{gather}
q_{\bar{\ga}} = 0 \text{ on } \Si,
\nonumber
\end{gather}
which then implies, upon employing (\ref{decomp_F_lemma_2}) one more time, that
\begin{gather}
F_{\bar{\al}\bar{\be}\bar{\ga}} = 0 \text{ on } \Si,
\nonumber
\end{gather}
where (\ref{use_constraints_vanish_F}) has been used.

Finally, using the Gauss equation and arguing as in proposition 
\ref{prop_initial_data}, we conclude that the constraints implied by the
decomposition of the Riemann tensor (i.e., (\ref{decomposition}) 
with $d^\al_{\mss \be\ga\de} = 0$) are satisfied on $\Si$. This, combined
with (\ref{d_reduced}), gives the vanishing of $d^\al_{\mss \be\ga\de}$ on 
$\Si$.

We conclude that the quantities (\ref{quantities_system_sub})
vanish on $\cM$. This implies, evoking (\ref{decomp_F_lemma_2})
once more, that $F_{\bar{\al} \bar{\be}\bar{\ga}} = 0$ also holds on $\cM$.
The remaining components in (\ref{vanishing_constraints})  also vanish
due to the identities
\begin{gather}
- 2 d^0_{\mss \bar{\al}  0 \bar{\be} } = \ve^{\bar{\mu}}_{\mss \bar{\al} \bar{\be} }
 \ve_{\bar{\mu}}^{\mss \bar{\si}\bar{\tau}} d^0_{\mss \bar{\si}0\bar{\tau}} ,
 \nonumber \\
 -2 d^0_{\mss \bar{\nu}\bar{\la}\bar{\mu}} = 
\ve_{\bar{\mu}\bar{\la}}^{\mss\mss \bar{\al}}
\ve_{(\bar{\nu}}^{\msm \bar{\be}\bar{\ga}} d^0_{\mss \bar{\al}) \bar{\be} \bar{\ga}} 
+  \pi_{\bar{\nu} \bar{\mu}}^{}  d^{\bar{\ga}}_{\mss 0 \bar{\la} \bar{\ga}}
-  \pi_{\bar{\nu} \bar{\la}}^{}  d^{\bar{\ga}}_{\mss 0 \bar{\mu} \bar{\ga}},
\nonumber
\end{gather}
which are verified by inspection with the help of 
(\ref{anti_sym_d_00}), (\ref{identity_ve_1}) and (\ref{identity_ve_2}).
\end{proof}

\begin{rema}
Notice that, as in the familiar case of wave coordinates,  propagation
of the gauge requires that the constraint equations be satisfied.
\end{rema}

We already know that the connection associated with $\Ga_{\al \ms \be}^{\ms \ga}$
is metric. Since it is also torsion-free by 
proposition \ref{prop_constraints_vanish}, we obtain:

\begin{coro}
The connection defined by the coefficients $\Ga_{\al \ms \be}^{\ms \ga}$ is the Levi-Civita
connection of the metric $g$.
\label{coro_Levi}
\end{coro}

\subsection{Solution to the original system.}

Showing that the solution $z$ of the reduced system yields a solution 
to the original Einstein-Euler-Entropy system is now a matter of unwrapping
all our definitions. 

\begin{prop}
Let $W^\al_{\mss \be\ga\de}$ and $S_{\al \be}$ be as in section
\ref{propagation_gauge} and $g$ the metric constructed out of the solution
of the reduced system given in proposition \ref{prop_well_pos_reduded}.
Then $W^\al_{\mss \be\ga\de}$ and $S_{\al \be}$ are, respectively,
the Weyl and the Schouten tensor of the metric $g$. Furthermore, 
equations (\ref{EEE_frame_system}) are satisfied, and the quantity $s_\al$ from
proposition \ref{prop_well_pos_reduded} is in fact the derivative of $s$.
\label{prop_EEE_frame_solutions}
\end{prop}
\begin{proof}
The connection given by $\Ga_{\al \ms \be}^{\ms \ga}$ is the Levi-Civita
connection of $g$ by corollary \ref{coro_Levi}. Thus, if we denote by
$\widehat{ W}^\al_{\mss \be\ga\de}$ and $\widehat{S}_{\al \be}$ 
the Weyl and the Schouten tensor of the metric $g$, we have 
\begin{gather}
R^\al_{\ms\be\ga\de} =  \widehat{ W}^\al_{\ms \be \ga \de} +
g^\al_{\ms [\ga} \widehat{ S }_{\de]\be}^{} - g_{\be[\ga}^{} \, \widehat{ S }_{\de]}^{\ms\al}
\nonumber
\end{gather}
which in turn equals
$  W^\al_{\ms \be \ga \de} +
g^\al_{\ms [\ga} S_{\de]\be}^{} - g_{\be[\ga}^{}S_{\de]}^{\ms\al}$
since $d^\al_{\mss \be\ga\de} = 0$ by proposition \ref{prop_constraints_vanish}.
Hence, tracing the equality 
\begin{gather}
\widehat{ W }^\al_{\ms \be \ga \de} +
g^\al_{\ms [\ga} \widehat{ S }_{\de]\be}^{} - g_{\be[\ga}^{} \widehat{ S }_{\de]}^{\ms\al} 
= 
  W^\al_{\ms \be \ga \de} +
g^\al_{\ms [\ga} S_{\de]\be}^{} - g_{\be[\ga}^{}S_{\de]}^{\ms\al}
\label{Weyl_Schouten}
\end{gather}
and using that $W$ is traceless in light of corollary \ref{coro_trace_eq_state},
we obtain
that $S_{\al\be}$ is indeed the Schouten tensor, which 
then implies $W^\al_{\ms \be \ga \de} = \widehat{ W }^\al_{\ms \be \ga \de}$
by using (\ref{Weyl_Schouten}) again. Then (\ref{Einstein_eq_equivalent}), with $S_{\al\be}$ being the Schouten tensor, also holds.
The remaining equations 
of (\ref{EEE_frame_system}) are satisfied by propositions \ref{well_posedness_reduced_section}
and \ref{prop_constraints_vanish}, and corollary \ref{coro_trace_eq_state}. Notice that these results show 
the validity of the (\ref{EEE_frame_system}) in fluid source gauge, but by the tensorial
nature of the equations, they hold in any frame.

Put  $\widehat{s}_\al = \nabla_\al s$. Then, since $u^\al \nabla_\al s = 0$ by (\ref{reduced_eq_s}), we obtain $\cL_u \widehat{s}_\al = 0$, where $\cL$ is the Lie derivative,
which in turn implies
\begin{gather}
\partial_t \widehat{s}_\al - (\Ga_{0 \msm \al}^{\ms \mu} - \Ga_{\al \msm 0}^{\ms \mu} ) \widehat{s}_\mu = 0 .
\nonumber
\end{gather}
From (\ref{reduced_eq_s_al}) and the construction of the initial data (proposition
\ref{prop_initial_data}), it  follows that $\widehat{s}_\al = s_\al$.
\end{proof}

\noindent \emph{Proof of theorem \ref{main_theorem}:} 
By proposition \ref{prop_EEE_frame_solutions} and (\ref{div_T_q}) we obtain that
the equations of the Einstein-Euler-Entropy system are satisfied. Notice that 
$u^\al u_\al =1$ holds by the way the solution was constructed; so 
$\nabla^\mu T_{\mu \al} = 0$ in fact implies 
(\ref{Euler_1}) and (\ref{Euler_2}). That $\cM \approx [0,T_E] \times \Si$
is indeed an Einsteinian development follows from the fact that the constraint
equations are satisfied, and by construction $(\cM, g, u)$ is a perfect fluid source.

Let  $\{ e_\mu \}_{\mu = 0}^3$ be a perfect fluid source gauge, with
the coordinates $\{ x^A \}_{A=0}^3$ arranged as explained below definition
\ref{def_fluid_source_gauge}. Since $e^{\bar{A}}_{\ms \bar{\al}}$ 
is in 
$C^0([0,T_E], H_{ul}^{s+1}(\Si))$ $\cap 
C^1([0,T_E], H_{ul}^{s}(\Si)) $ and $e^A_{\ms 0} = \de^A_{\ms 0}$,
by 
(\ref{def_metric_frame_coeff})
we conclude that the inverse of the one-parameter family of
 metrics $g_t$ induced on
$\Si_t = \{ t = \text{ constant} \}$ belongs to 
$C^0([0,T_E], H_{ul}^{s+1}(\Si)) \cap 
C^1([0,T_E], H_{ul}^{s}(\Si))$, and so does $g_t$ itself because $s > \frac{3}{2} + 2$.
Since $d^\al_{\mss \be\ga\de} = 0$ and 
$E_{\bar{\al}\bar{\be}},\, B_{\bar{\al}\bar{\be}}  \in 
 C^0([0,T_E], H_{ul}^{s+1}(\Si)) \cap 
C^1([0,T_E], H_{ul}^{s}(\Si)) \cap C^2([0,T_E], H_{ul}^{s-1}(\Si))$, by 
(\ref{decomp_W}), (\ref{E_0_B_0_gauge}), (\ref{E_bar_B_bar_gauge}),
propositions
\ref{prop_constraints_vanish} and \ref{prop_well_pos_reduded}, we conclude that
$g \in  C^0([0,T_E], H_{ul}^{s+1}(\Si)) \cap 
C^1([0,T_E],$ $ H_{ul}^{s}(\Si)) \cap C^2([0,T_E],H_{ul}^{s-1}(\Si))$. 
The four-velocity $u$ belongs to 
$C^0([0,T_E],$ $ H_{ul}^{s}(\Si))\cap 
C^1([0,T_E], H_{ul}^{s-s}(\Si))$ and has the correct projection because
of (\ref{v_g}) and proposition \ref{prop_well_pos_reduded}.

The functions $r$ and $s$ have the desired regularity and take the correct 
initial values by the way they have been constructed, and $\varrho$ and $p$ are
given by (\ref{eq_of_state}) and (\ref{pressure}) as a consequence of corollary
\ref{coro_trace_eq_state}, which implies that $\nu^2$ also has the correct form.
By continuity on the time variable and the hypotheses of the theorem, we
obtain that $r>0$ and $\nu^2 > 0$ for small  $T_E$. We cannot have
$ s(p) < 0$ for a point $p$ near the initial Cauchy surface because this would contradict
$u^\al \nabla_\al s =0$ and $\scr_0 \geq 0$,  hence $s \geq 0$.
\hfill $\qed$.

\begin{rema}
The less regular frame coefficients $e^0_{\ms \bar{\al}}$ do not affect
the regularity of the space-time metric because they contribute only to the
mixed entries  $g_{0\bar{A}}$, which are 
gauge terms  therefore having no direct physical or geometrical
meaning.
\end{rema}

\section{Further remarks.\label{further}}

The case of barotropic fluids is treated as a particular case of 
theorem \ref{main_theorem}, at least as long as 
(\ref{sound_speed_positive}) and (\ref{enthalpy_positive}) hold.
Although the condition $\nu^2 > 0$ is violated by pressure-free matter,
in this situation, equations (\ref{EEE_frame_system}) simplify considerably;
 a reduced system which does not require the introduction of $\nu^2$ 
can be derived \cite{Fri}, and a system for the propagation of the gauge,
which does not involve $\nu^2$ either can also be constructed \cite{FriRen}.
The arguments of section \ref{proof_main_theorem_section} can then
be reproduced, yielding a statement analogous to theorem 
\ref{main_theorem} for pressure-free matter. Notice also that 
when a fluid is isentropic, the entropy can be treated
as a parameter in the equation of state, and the equations of motion take the form
of  those of a barotropic fluid. 

It should also be noticed that condition (\ref{sound_speed_causal}) has never been
used. In fact, such inequality is necessary due to causality, but it plays
no role on the well-posedness of the Einstein-Euler-Entropy system. There
is at least one situation where it may be desirable to consider equations
of state where $\nu^2 \leq 1$ is not satisfied, namely, the construction 
of appropriate gauge conditions for the vacuum Einstein equations. In this
situation, setting $\K = 0$ in our system, the Euler equations decouple, and 
the role of the four-velocity $u$ is to fix the gauge. 
As Friedrich has pointed out \cite{Fri}, this procedure 
may be particularly important in numerical treatments of Einstein equations,
where ever more sophisticated gauge choices are crucial for accurate results.

Finally, we remark that although theorem \ref{main_theorem} does not
provide an existence result for the fluid body discussed in the introduction, 
when one is  interested solely in its behavior near 
a small compact set $\Om \subset \Si$, the hypothesis that $r_0$ is uniformly bounded
away from zero can be replaced by the condition that $r_0$ is positive
in $\Om$ and decays sufficiently fast on its complement. As 
mentioned in section \ref{basic_setting}, 
this will not generally yield a uniform time span for the solution,
but a uniform $T_E>0$ will exist in the neighborhood of $\Om$.
Whether this suffices for studying the dynamics of the fluid near 
 $\Om$ will obviously depend on the particular application
one has in mind.

\appendix

\section{The uniformly local Sobolev spaces.\label{ul_Sobolev_appendix}}

Below we review the  uniformly local Sobolev spaces
originally introduced by Kato \cite{KatoQL}. Some of their properties
can also be found in \cite{C}.

Let $(M,\ga)$ be a Riemannian manifold. For any open set $U \subseteq M$,
let $H^s(U, \ga)$ be the Sobolev space of tensors of a given rank defined on
$U$, with the derivative $\nabla$ and the measured of integration
$\mu$
used to define the Sobolev norm being those of the metric $\ga$.
Recall that  $H^s_{loc}(M, \ga)$ is defined as the space of tensor 
fields that belong to $H^s(U,\ga)$ for any relatively compact $U \subseteq M$.

\begin{defi}
Let $\{ U_j \}$ be a locally finite covering of $M$ by relatively compact
open sets $U_j$. We define $H^s_{ul}(M,\ga)$ as the space of 
tensor fields $f \in H^s_{loc}(M,\ga)$  such that
\begin{gather}
\p f \p_{H^s_{ul}} : = \sup_j \p f \p_{H^s(U_j, \ga)} < \infty.
\nonumber
\end{gather}
\end{defi}

$H^s_{ul}(M,\ga)$  is a Banach space with the above norm.
$H^s_{ul}(M,\ga)$ will have the usual embedding and multiplication
properties of Sobolev spaces provided  all of the $H^s(U_j, \ga)$ have them.
This will be the case when  $(M,\ga)$ has injectivity radius bounded from below
away from zero.
In this case, it also holds that the Sobolev constants relative to the subsets
$U_i$ are uniformly bounded, and $\left.  \ga\right|_{U_i}$ is uniformly equivalent
to the Euclidean metric.

\section{Derivation of the reduced system.\label{derivation_appendix}}

In this appendix, it is shown how equations (\ref{reduced_system}) 
are obtained from (\ref{EEE_frame_system}) after making a gauge choice. 
This has been done first by Friedrich in \cite{Fri}, which the reader is referred
to for more details.

Assume that a solution to the Einstein-Euler-Entropy system in the frame
formalism, equations (\ref{EEE_frame_system}), is given. 

In fluid source gauge, the vanishing of $\cT_{\,0 \ms \bar{\al}}^{\mss \mu}$ 
corresponds  
to  equation (\ref{reduced_eq_frame}), while
(\ref{definition_Riemann}), (\ref{decomp_W}) and $d^\al_{\mss \be\ga\de} = 0$
give (\ref{reduced_eq_Ga_bar}).

It follows from our definitions that
\begin{align}
\begin{split}
& \frac{1}{p+\varrho} \nabla_{[\al}q_{\be]} = 
 u^\mu \nabla_\al \nabla_\mu u_\be - u^\mu \nabla_\be \nabla_\mu u_\alpha
 - \nu^2 u_\be \nabla_\al \nabla_\mu u^\mu + \nu^2 u_\al \nabla_\be \nabla_\mu u^\mu
 \\
 & -\nu^2 \nabla_\mu u^\mu \nabla_\al u_\be 
  + \nu^2 \nabla_\mu u^\mu \nabla_\be u_\al
 + \nabla_\al u^\mu \nabla_\mu u_\be  - \nabla_\be u^\mu \nabla_\mu u_\al
 \\
 &
 + \frac{p+\varrho}{\nu^2} \left( \frac{\partial p }{\partial \varrho^2} \right)_s
 \nabla_\mu u^\mu   \Big [ u_\al u^\la \nabla_\la u_\be 
 - u_\be u^\la \nabla_\la u_\al \Big ] \\
 &  +\Big [ \frac{\partial \nu^2}{\partial s} - \frac{1}{p+\varrho}
 \Big ( 1 + \frac{r}{\nu^2}\frac{\partial \nu^2}{\partial r} \Big ) \frac{\partial p}{\partial s}
 + \frac{1}{\varrho + p } \nu^2 \frac{\partial \varrho}{\partial s} 
 \Big ] \Big( u_\al \nabla_\mu u^\mu s_\be - u_\be \nabla_\mu u^\mu s_\al \Big ) 
 \\
 & + \frac{1}{p + \varrho} \Big( \frac{\partial \varrho}{\partial s} - \frac{1}{\nu^2} \frac{\partial p}{\partial s} \Big ) \Big( s_\al u^\mu \nabla_\mu u_\be  - s_\be u^\mu \nabla_\mu u_\be  \Big ).
 \end{split}
 \nonumber
\end{align}
Combining the quantities in the system (\ref{EEE_frame_system}) with the above expression,
we obtain that in fluid source gauge equations (\ref{reduced_eq_Ga_0_0}) and 
(\ref{reduced_eq_Ga_0}) correspond to
\begin{gather}
\nu^2 d^{\bar{\mu}}_{\mss 0 \bar{\al} \bar{\mu}} - \frac{1}{p + \varrho}
\nabla_{[0} q_{ \bar{\al} ] } = 0
\nonumber
\end{gather}
and 
\begin{gather}
\nu^2 d^{0}_{\mss  \bar{\al} 0 \bar{\be}} + \frac{\nu^2 }{p + \varrho}
\nabla_{[\bar{\al} } q_{ \bar{ \be } ] } = 0,
\nonumber
\end{gather}
respectively.

Next, notice that equations (\ref{reduced_eq_rho}), (\ref{reduced_eq_s})
and (\ref{reduced_eq_r}) correspond to (\ref{Euler_1}), (\ref{loc_adiabatic})
and (\ref{rest_mass_conservation}) when written in fluid source gauge,
whereas (\ref{reduced_eq_s_al}) is the same as $\cL_u \nabla_\al s = 0$, which 
is implied by equations (\ref{EEE_frame_system}) ($\cL$ is the Lie derivative).

Finally, we have the decomposition
\begin{gather}
\nabla^\mu T_{\mu \al} = w u_\al + q_\al 
\nonumber
\end{gather}
where $w$ is given by the left-hand side of (\ref{Euler_1}). From (\ref{div_Friedrichs_tensor})
and (\ref{Einstein_eq_equivalent}), we obtain
\begin{align}
-\frac{1}{2} \K q_\al = 
\pi_\mu^{\mss \si} \pi_\nu^{\mss \tau} \pi_\al^{\mss \la}  F_{\si\tau\la} \pi^{\mu\nu}
- \pi_\al^{\mss \la} F_{\mu \la \nu} u^\mu u^\nu,
\nonumber
\end{align}
and
\begin{gather}
\frac{1}{2} \K w = 
\pi_\mu^{\mss \si}  \pi_\nu^{\mss \la}  F_{\si\tau\la} u^\tau \pi^{\mu\nu}.
\nonumber
\end{gather}
With the help of these expressions 
and (\ref{decomp_W}) and defining $\D$ as in (\ref{spatial_derivative}), 
it  follows that
\begin{align}
\begin{split}
& 
\pi_{(\al}^{\ms \mu} \pi_{\be)}^{\ms \nu} F_{\mu \tau \nu} u^\tau
- \frac{1}{3} \pi_{\al\be} \pi^{\mu\nu} \pi_\mu^{\mss \si} \pi_\nu^{\mss \la}
F_{\si \xi \la} u^\xi = 
u^\mu \nabla_\mu E_{\al\be} + E_{\mu\be}\nabla_\al u^\mu+ 
\\
&
\hspace{0.5cm}
E_{\mu\al}\nabla_\be u^\mu + \D_\mu B_{\nu (\al}^{} \ve_{\be)}^{\msm \mu \nu} 
-2 u^\la \nabla_\la u_\mu \ve^{\mu \nu}_{\mss\mss(\al} B^{}_{\be) \nu}
-3  \pi_{(\al}^{\ms \mu} \nabla_{|\mu|} u^\la E^{}_{\be)\la} 
\\
&
\hspace{0.5cm}
 - 2 \pi^{\mu}_{\mss \la} \nabla^\la u_{(\al} E^{}_{\be)\mu} 
+ \pi_{\al\be}  \pi^{\mu\la}\nabla_\la u^\nu E_{\mu\nu}
+ 2 \pi^{\mu\nu}\pi_{\mu}^{\mss \la}\nabla_\la u_\nu E_{\al \be}
\\
& 
\hspace{0.5cm}
+ \frac{1}{2} \K (p + \varrho) \Big ( \pi_{(\al}^{\ms \la }\nabla_{|\la|} u^{}_{\be)}
- \frac{1}{3} \pi^{\mu\nu}\pi_{\mu}^{\mss \la}\nabla_\la u_\nu \pi_{\al\be} \Big ) ,
\end{split}
\nonumber
\end{align}
which, in fluid source gauge, gives (\ref{reduced_eq_E}). The equation
(\ref{reduced_eq_B}) is obtained by a similar argument
after computing 
\begin{gather}
\ve_{(\al}^{\msm \mu \nu} \pi_{\be)}^{\msm \la} \pi_\mu^{\mss \si}
\pi_\nu^{\tau} F_{\la \si \tau}.
\nonumber
\end{gather}

\end{document}